\documentclass[11pt]{amsart}
\usepackage{fullpage,amssymb}
\usepackage{mathrsfs}

\usepackage[T1]{fontenc}

\usepackage{color}   %May be necessary if you want to color links
\usepackage{hyperref}
\hypersetup{
    linktoc=all,     %set to all if you want both sections and subsections linked
    linkcolor=black,  %choose some color if you want links to stand out
}

\newtheorem{theorem}{Theorem}[section]
\newtheorem{corollary}[theorem]{Corollary} 
\newtheorem{lemma}[theorem]{Lemma}
\newtheorem{proposition}[theorem]{Proposition}

\theoremstyle{definition}
\newtheorem{definition}[theorem]{Definition}
\newtheorem{remark}[theorem]{Remark}
\newtheorem{example}[theorem]{Example}
\newtheorem{problem}[theorem]{Problem}

\newtheorem{notation}[theorem]{Notation}
\newtheorem{claim}[theorem]{Claim}

\newcommand{\spa} {{\rm span}}
\newcommand{\SING} {{\rm SING}}
\newcommand{\NSING} {{\rm NSING}}

\newcommand{\R}{{\mathbb R}}

\newcommand{\C}{{\mathbb C}}
\newcommand{\N}{{\mathbb N}}

\newcommand{\Q}{{\mathbb Q}}
\newcommand{\Z}{{\mathbb Z}}

\newcommand{\Mat}{\operatorname{Mat}}
\newcommand{\GL}{\operatorname{GL}}
\newcommand{\SL}{\operatorname{SL}}
\newcommand{\ST} {\operatorname{ST}}
\newcommand{\Lie} {{\rm Lie}}
\newcommand{\gl}{{\mathfrak{gl}}}

\newcommand{\sll}{{\mathfrak{sl}}}
\newcommand{\sgn} {{\rm sgn}}
\newcommand{\cd} {\star}

\newcommand{\llangle}{\,\,(\!\!\!\!<\!}
\newcommand{\rrangle}{\!>\!\!\!\!)\,\,}

\newcommand{\g} {\mathfrak{g}}

\newcommand{\h} {\mathfrak{h}}

\newcommand{\End}{{\rm End}}

\newcommand{\Supp} {{\rm Supp}}
\newcommand{\USupp} {{\rm USupp}}

\title{Singular tuples of matrices is not a null cone \\ (and, the symmetries of algebraic varieties)}
\author{Visu Makam} \thanks{Visu's research was supported by NSF grant No. DMS -1638352 and NSF grant No. CCF-1412958.}
\address{School of Mathematics, Institute for Advanced Study, Princeton}
\email{visu@ias.edu}

\author{Avi Wigderson} \thanks{
Avi's research was supported by NST grant No. CCF-1412958.}
\address{School of Mathematics, Institute for Advanced Study, Princeton}
\email{avi@ias.edu}

\thanks{}
\begin{document}

\maketitle
\begin{abstract}
The following multi-determinantal algebraic variety plays a central role in algebra, algebraic geometry and computational complexity theory: $\SING_{n,m}$, consisting of all $m$-tuples of $n\times n$ complex matrices which span {\em only} singular matrices. In particular, an efficient deterministic algorithm testing membership in $\SING_{n,m}$ will imply super-polynomial circuit lower bounds, a holy grail of the theory of computation.

A sequence of recent works suggests such efficient algorithms for memberships in a general class of algebraic varieties, namely the {\em null cones} of linear group actions. Can this be used for the problem above? Our main result is negative: $\SING_{n,m}$ is {\em not} the null cone of any (reductive) group action! This stands in stark contrast to a non-commutative analog of this variety, and points to an inherent structural difficulty of $\SING_{n,m}$.

To prove this result we identify precisely the group of symmetries of $\SING_{n,m}$. We find this characterization, and the tools we introduce to prove it, of independent interest. Our work significantly generalizes a result of Frobenius for the special case $m=1$, and suggests a general method for determining the symmetries of algebraic varieties.
\end{abstract}

\tableofcontents

%\begin{center}
%For this entire paper, the ground field will be $\C$, the field of complex numbers.
%\end{center}

\section{Introduction} \label{sec:intro}
We start the introduction with a general discussion of the main problems and their motivations. Next we turn to describe our main object of study - singular spaces of matrices. We end by formally stating our main results. While a few technical terms here may be unfamiliar to some readers, we will have a simple running example to demonstrate all essential notions. Throughout, the underlying field is the complex numbers $\C$. 
\subsection{Motivation and main problems}
Consider a  (reductive\footnote{A technical term that includes all classical groups.}) 
group $G$ acting (algebraically) on a vector space $V$ by linear transformations. Understanding this very general setting is the purview of {\em invariant theory}. As a simple, and very relevant running example, consider the following.

\begin{example} [Running example] \label{E-running} 
Consider $G= \SL_n$ acting on $n\times n$ matrices (namely $V=\C^{n^2}$) by left multiplication, i.e, the action of $P \in \SL_n$ sends the matrix $X$ to $PX$.
\end{example}

A group action partitions $V$ into {\em orbits}: the orbit of $v$ is the set of all points in $V$ it can be moved to by an element $g\in G$. An even more natural object in our setting is the {\em orbit closure}: all limit points of an orbit\footnote{where limits can be equivalently  taken in the Euclidean or Zariski topology}. 

The {\em null cone} of a group action is the set of points $v \in V$ whose orbit closure contains the origin, namely the point $0$. Null cones of group actions are central to invariant theory, and are interesting algebraic objects to study in mathematics and physics. More recently, connections to fundamental problems in computational complexity have surfaced. Diverse problems (see \cite{GGOW16, BGFOWW}) such as bipartite matching, equivalence of non-commutative rational expressions, tensor scaling and quantum distillation, can each be formulated (for specific choices of $G,V$ and an action) as a {\em null cone membership problem} -- given a point $v \in V$, decide if it is in the null cone. Note that in our running example, i.e., Example~\ref{E-running}, the null cone is precisely the set of singular matrices.

A closely related problem is the {\em orbit closure intersection problem} -- given $v,w \in V$, decide if the orbit closure of $v$ and $w$ intersect. The orbit closure intersection problem is a generalization of the null cone membership problem, and this too has many connections with arithmetic complexity. For example, the graph isomorphism problem can be phrased as an orbit closure intersection problem! We refer to \cite{GCTV} for more details on the aforementioned problems and their relevance in the Geometric Complexity Theory (GCT) program, which is an algebro-geometric approach to the VP vs VNP problem (an algebraic analog of P vs NP). Note that in Example~\ref{E-running}, the orbit closure of two matrices $X$ and $Y$ intersect precisely when $\det(X)=\det(Y)$.

In an exciting series of recent works, efficient algorithms for the null cone membership and orbit closure intersection problems in various cases have been discovered, and moreover techniques have developed that may allow significant generalization of their applicability \cite{GGOW16, IQS2, FS, DM, DMOC, GGOW18,AGLOW18,BGFOWW,BGFOWW2,CF, DM-arbchar, DM-siq}. Curiously, Geometric Complexity Theory (morally) predicts efficient algorithms for null cone membership problems in great generality (see \cite{GCTV} for precise formulations), although establishing this remains an elusive goal. 

What is remarkable is the possibility that such efficient algorithms, through the work of \cite{KI}, may enable proving non-trivial {\em lower bounds} on computation, the major challenge of computational complexity. Specifically, what is needed is a deterministic polynomial time algorithm for a problem called Symbolic Determinant Identity Testing (SDIT)\footnote{A canonical version of the Polynomial Identity Testing (PIT) problem.} that is central to this work, and will be defined soon. SDIT happens to be a membership problem in an {\em algebraic subvariety}, a context generalizing null cones.

A subset $S \subseteq V$ is called an algebraic subvariety\footnote{We do not require irreducibility in our definition of varieties.} (or simply a {\em subvariety}) if it is the zero locus of a collection of polynomial functions on $V$. Many algorithmic problems can be phrased as ``membership in a subvariety'', and is non-trivial when the underlying set of polynomials is given implicitly or are difficult to compute. It is a fundamental result of invariant theory that {\em every null cone is an algebraic subvariety}, a connection which goes through {\em invariant polynomials} of group actions. A polynomial function $f$ on $V$ is called invariant if it is constant along orbits, i.e., $f(gv) = f(v)$ for all $g \in G, v \in V$. Invariant polynomials form a graded subring of $\C[V]$, the ring of polynomial functions on $V$. Mumford proved that the orbit closures of any two points $v,w \in V$  intersect, if and only if $f(v) = f(w)$ for all invariant polynomials\footnote{Reductivity is essential for this.}, see \cite{Mumford}. As a consequence, the null cone can also be described as the zero locus of all (non-constant) homogenous invariant polynomials. Indeed, this analytic-algebraic connection provides the path to structural and algorithmic understanding of the null cone membership and orbit closure intersection problems  via invariant theory.

Summarizing,  if a subvariety $S \subseteq V$ happens to be a null cone for some group action, then the aforementioned algorithms can be used to decide ``membership in $S$'', with the exciting possibility that they could very well be efficient. Of course, not every subvariety is a null cone, which leads to the following interesting problem:

\begin{problem} \label{prob.nullcone}
Given a subvariety $S \subseteq V$, is it the null cone for the (algebraic) action of a (reductive) group $G$ on $V$?
\end{problem}

\begin{remark}
We specifically refer to $S$ as a {\em subvariety} of $V$ rather than just call it a {\em variety} for the following reason. In the above problem, it is important that we view $S$ as a subset of $V$. As an abstract variety, a different embedding of $S$ into another vector space could very well make it a null cone.\footnote{For example, if we consider the parabola described as the zero locus of $y - x^2$ in $\C^2$, this is not a null cone (because null cones are stable under scalar multiplication, but the parabola isn't. However, as a variety, this is just the affine line, which is definitely a null cone (for the action of $\C^*$ on $\C$ by multiplication).} Our setting of a given embedding makes the problem well-defined. 
\end{remark}

We now make an important observation. If $S$ is to be the null cone for the action of a group $G$, then the group must ``preserve'' $S$, i.e., for all $g \in G$, we must have $gS = S$. We define the {\em group of symmetries} to be the (largest) subgroup of $\GL(V)$ consisting of all linear transformations that preserve $S$. With reference to Example~\ref{E-running}, one might ask which is the largest group of symmetries in $\GL_{n^2}$ which preserves the set $n\times n$ the singular matrices (which is defined by the zeros of the single determinant polynomial). This question was resolved by Frobenius \cite{Frob} as we will later see, and is a very special case of our main technical result. 

So, the (hypothetical) acting group $G$ must be a subgroup\footnote{Any group $G$ acting on $V$ gives a map $\rho:G \rightarrow \GL(V)$. The null cone for $G$ is the same as the null cone for $\rho(G)$, so we can always restrict ourselves to subgroups of $\GL(V)$ when concerned about Problem~\ref{prob.nullcone}. Moreover, note that if $G$ is reductive, so is $\rho(G)$.} of the group of symmetries of $S$. Roughly speaking, this provides an important ``upper bound'' to the groups that one must consider while resolving Problem~\ref{prob.nullcone}. 
 
\begin{problem} \label{prob.gos}
Given a subvariety $S \subseteq V$, compute its group of symmetries.
\end{problem}

Needless to say, the important role of symmetries in mathematics in present just about every branch, and exploiting symmetries is an immensely powerful tool. Specifically, the fact that the determinant and permanent polynomials are {\em defined} by their symmetries form the starting point to the GCT of Mulmuley and Sohoni \cite{MS01,MS02} program mentioned earlier towards the VP $\neq$ VNP conjecture. Computing the group of symmetries of an algebraic variety is an extremely natural problem (even in the absence of Problem~\ref{prob.nullcone}!), and may be useful for other purposes. We now elaborate informally on the path we take to solve Problem~\ref{prob.gos}, and another natural problem it raises.

The group of symmetries of an algebraic subvariety $S \subseteq V$ is always an algebraic subgroup of $\GL(V)$ (and hence a Lie subgroup). Suppose that $H$ is an algebraic group that acts linearly on a vector space $V$. It is a fact that the null cone for the action of its identity component\footnote{The identity component is the connected component of $H$ that contains the identity element. It is always an algebraic subgroup.} (denoted $H^\circ$) is the same as the null cone for the action of $H$. Thus, for Problem~\ref{prob.nullcone}, one might as well study the {\em connected group of symmetries}, i.e., the identity component of the group of symmetries. Indeed, if $S$ is the null cone for the action of a reductive group $G$, then it is the null cone for the action of its identity component $G^\circ$, which must be a subgroup of the connected group of symmetries. Thus, we are led to problem below.

\begin{problem} \label{prob.cgos}
Given a subvariety $S \subseteq V$, compute its connected group of symmetries.
\end{problem}

To understand that Problem~\ref{prob.cgos} really is much easier than Problem~\ref{prob.gos}, one needs to realize that connected group of symmetries is a connected algebraic subgroup of $\GL(V)$, and so in particular is determined by its Lie algebra (which is a Lie subalgebra of the Lie algebra of $\GL(V)$). Roughly speaking, we will use this to ``linearize'' the problem.

Algebraic subvarieties are defined as the zero locus of a collection of polynomials. Suppose we have a  collection of homogeneous polynomials $\{f_i:i \in I\}$, and let $S$ be its zero locus. If the ring of invariants for the action of some group $G$ is precisely $\C[f_i : i \in I]$, then $S$ would be the null cone (recall that the null cone can be seen as the zero locus of non-constant homogeneous invariant polynomials). This brings us to another interesting problem, which can be seen as a scheme-theoretic version of Problem~\ref{prob.nullcone}

\begin{problem} \label{prob.invring}
Given  a collection of polynomials $\{f_i: i \in I\}$ on $V$, is there a group $G$ acting on $V$ by linear transformations such that the ring of invariants is $\C[\{f_i: i \in I\}]$.
\end{problem}

Curiously, the above problem is in some sense is an inverse problem to the classical one in invariant theory: there, given a group action on $V$, we seek its invariant polynomials, whereas here we are given the polynomials, and seek the group which makes them all invariant. 

\begin{remark}
Both Problem~\ref{prob.gos} and Problem~\ref{prob.invring} belong to a general class of problems called {\em linear preserver problems}. We refer the reader to the survey \cite{LP-survey} which contains in particular some general techniques for approaching linear preserver problems. These techniques do not seem to be sufficient for us.
\end{remark}

Finally, let us mention that all the aforementioned problems are very natural, interesting in their own right, and could potentially use tools from invariant theory, representation theory, Lie theory, algebraic geometry, commutative algebra and computational complexity. 

\subsection{The algebraic variety SING and the computational problem SDIT}
Having introduced the problems of interest, let us introduce the subvariety which we will be the main focus of this paper. 
%\Anote{As far as I can tell, we never use rectangular matrices. Perhaps we can just work with $Mat_n$ and $Mat_n^m$ throughout?} 
%\Anote{Do we ever need rectangular matrices? Why not work with $M_n$ throughout}
Let $\Mat_{n}$ denote $n \times n$ matrices with entries in $\C$. Let $t_1,\dots,t_m$ be indeterminates, and let $\C(t_1,\dots,t_m)$ denote the function field in $m$ indeterminates. Define
\begin{align}
\SING_{n,m} \triangleq \left\{X = (X_1,\dots,X_m)\in \Mat_{n}^m\ |\ \sum_{i=1}^m t_i X_i \text{ singular (over $\C(t_1,\dots,t_m)$)}\right\}.
\end{align}

Note that $\SING_{n,m} \subseteq V=\C^{mn^2}$, given by the zero locus of all polynomials $\{\det(c_1X_1 + c_2X_2+ \dots + c_mX_m)\,:\, c_i \in \C\}$. While this is an uncountable set, one can easily make it finite. Another important note is that the case $m=1$ is the null cone for our simple running example (Example~\ref{E-running}) of the previous subsection!

The subvariety $\SING_{n,m}$ is of central importance in computational complexity. The membership problem for $\SING_{n,m}$ (i.e., given $X \in  \Mat_{n}^m$, decide if $X \in \SING_{n,m}$) is often called Symbolic Determinant Identity Testing (SDIT). This problem SDIT is also sometimes referred to as the {\em Edmonds' problem}, as Edmond' paper ~\cite{Edm67} first explicitly defined it and asked if it has a polynomial time algorithm. Note that any {\em fixed} tuple $X = (X_1,\dots,X_m) \in \SING_{n,m}$ if and only if the {\em symbolic} determinant $\det(t_1X_1 + t_2X_2+ \dots t_mX_m)$ vanishes identically when viewed now as a polynomial in the new variables $t_1,\dots,t_m$. This viewpoint immediately provides an efficient {\em probabilistic}  algorithm for the SDIT \cite{Lovasz'79}: given $X$, simply pick (appropriately) at random values for the variables $t_i$ and evaluate the resulting {\em numeric} determinant. 

The importance of determining the complexity of SDIT stems from several central results in arithmetic complexity and beyond. First, Valiant's completeness theorem for VP \cite{Valiant} implies that SDIT captures the general problem of Polynomial Identity Testing (PIT) problem (see the survey \cite{SY}, for background and status of this problem, and more generally on arithmetic complexity). An equivalent way of phrasing Valiant's result is that SDIT {\em is} the {\em word problem} for $\C(t_1,\dots,t_m)$, namely testing if a rational expression in $\C(t_1,\dots,t_m)$ is identically zero.  A second, and far more surprising result we already mentioned, of Kabanets and Impagliazzo (see \cite{KI}), shows that efficient {\em deterministic} algorithms for PIT would imply circuit lower bounds, a holy grail of complexity theory. SDIT also plays an important role in the GCT program, see \cite{GCTV}. Finally, the structural study of the variety $\SING_{n,m}$, namely of singular spaces of matrices is a rich subject in linear algebra and geometry (see e.g. \cite{FR, EH,RW,RM1,RM2,RM3}).

It is illustrative to compare with the non-commutative version of the above story, and we will do so. Let $t_1,\dots,t_m$ be now {\em non-commuting} indeterminates, and let $\C\llangle t_1,\dots,t_m \rrangle$ denote the free skew field\footnote{The free skew field  is intuitively the natural non-commutative analog of $\C(t_1,\dots,t_m)$, namely may be viewed as the field of fractions completing non-commutative polynomials. However, we note that its very existence, let alone its construction is highly non-trivial, and was first established by Amitsur \cite{Am66} (see also \cite{Cohn}). For one illustration of the complexity of this field, it is easy to see that unlike in the commutative case, its elements cannot be represented as ratios of polynomials (or any finite number of inversions - an important result of \cite{Reutenauer}).}. Consider 
$$
\NSING_{n,m} \triangleq \left\{X = (X_1,\dots,X_m) \in \Mat_{n}^m\ | \ \sum_i t_iX_i \text{ singular (over $\C\llangle t_1,\dots,t_m \rrangle$)} \right\},
$$
which is clearly a non-commutative analog of $\SING_{n,m}$. Moreover, membership in $\NSING_{n,m}$ captures the word problem over the free skew field $\C\llangle t_1,\dots,t_m \rrangle$ (often called non-commutative rational identity testing (RIT)) in precisely the same manner as membership in $\SING_{n,m}$ captures the word problem over the function field $\C(t_1,\dots,t_m)$. 

The surprising fact is that membership in $\NSING_{n,m}$ {\em does} have polynomial time deterministic algorithms, see \cite{GGOW16, IQS2}. The main point to note is that the algorithms use crucially the fact that $\NSING_{n,m}$ is a null cone! Indeed, it is the null cone for the so called left-right action of $\SL_n \times \SL_n$ on $\Mat_{n}^m$ which is defined by:
$$
(P,Q) \cdot (X_1,\dots,X_m) = (PX_1Q^t,PX_2Q^t,\dots,PX_mQ^t),
$$
where $Q^t$ denotes the transpose of the matrix $Q$. In view of this, it is only natural to ask whether a similar story can be used to give an efficient algorithm for membership in $\SING_{n,m}$. This provides the principal motivation for studying Problem~\ref{prob.nullcone}.

For $S = \SING_{n,m}$, in this paper, we will answer Problem~\ref{prob.nullcone} and Problem~\ref{prob.gos} (and hence also Problem~\ref{prob.cgos}). Moreover, recall that $\SING_{n,m}$ is the zero locus of a natural collection of polynomials, namely $\{\det(\sum_i c_i X_i) : c_i \in \C\}$. We also give a negative answer to Problem~\ref{prob.invring} for this collection of polynomials. We will now proceed to give precise statements.

\subsection{Main results}
We begin by stating the main result, i.e., a negative answer to Problem~\ref{prob.nullcone} for $\SING_{n,m}$.

\begin{theorem} \label{theo:nullcone}
Let $n,m \geq 3$. Let $G$ be any reductive group acting algebraically on $\Mat_{n}^m$ by linear transformations. Then the null cone for the action of $G$ is not equal to $\SING_{n,m}$.
\end{theorem}

First, and foremost, let us observe that the condition $n,m \geq 3$ cannot be removed or even improved. Indeed, if
 $n \leq 2$ or $m \leq 2$, we have $\SING_{n,m} = \NSING_{n,m}$ and hence it is a null cone! Thus, the above theorem gives the strongest possible statement of this nature. The above theorem follows from the following one, which has no restrictions on $n$ and $m$.
 
\begin{theorem} \label{theo:nullcone2}
Let $G$ be any reductive group acting algebraically on $V = \Mat_{n}^m$ by linear transformations which preserve $\SING_{n,m}$ (i.e., $g \cdot \SING_{n,m} = \SING_{n,m}$ for all $g \in G$). Let $\mathcal{N} = \mathcal{N}_G(V)$ denote the null cone for this action. If the null cone $\mathcal{N} \subseteq\SING_{n,m}$, then the null cone $\mathcal{N} \subseteq \NSING_{n,m}$.
\end{theorem}
 
Indeed, Theorem~\ref{theo:nullcone} follows from the above theorem as $n,m \geq 3$ is precisely the condition needed to ensure that $\NSING_{n,m}$ is a proper subset of $\SING_{n,m}$.
 
A crucial component in the proof of the above theorem is the computation of the group of symmetries for $\SING_{n,m}$. The importance of this computation is well beyond the context of this paper. For example, it should serve as the starting point for any approach to SDIT that aims at utilizing symmetry. Let us formally define the group of symmetries for a subvariety.

\begin{definition} [Group of symmetries]
For a subvariety $S \subseteq V$, we define its group of symmetries
$$
\mathcal{G}_S = \{g \in \GL(V) \ |\ gS = S\}.
$$
The group of symmetries $\mathcal{G}_S$ is always an algebraic subgroup of $\GL(V)$. We call its identity component (denoted $\mathcal{G}_S^\circ$) the connected group of symmetries.
\end{definition}

In order to compute the group of symmetries for $\SING_{n,m}$, we first compute the connected group of symmetries. Viewing $\Mat_{n}^m$ as $\C^m \otimes \C^n \otimes \C^n$ elucidates a natural linear action of $\GL_m \times \GL_n \times \GL_n$ on $\Mat_{n}^m$. Concretely, the action is given by the formula: 
$$
(P,Q,R) \cdot (X_1,\dots,X_m) = \left(\sum_{j=1}^m p_{1j} QX_jR^{-1}, \sum_{j=1}^m p_{2j} QX_j R^{-1},\dots, \sum_{j=1}^m p_{nj} QX_j R^{-1}\right),
$$
where $p_{ij}$ denotes the $(i,j)^{th}$ entry of $P$. A linear action is simply a representation, so we have a map $\GL_m \times \GL_n \times \GL_n \rightarrow \GL(\Mat_{n}^m)$. We will call the image of this map $G_{n,m}$. 

\begin{theorem} \label{theo:cgos}
Let $S = \SING_{n,m} \subseteq V =  \Mat_{n}^m$. Then the connected group of symmetries $\mathcal{G}_S^\circ$ is the subgroup $G_{n,m}$.
\end{theorem}

We will discuss in the subsequent section, the strategy of proof in more detail for the above theorem. However, it is worth mentioning that it is essentially a linear algebraic computation on the level of Lie algebras, and is applicable in more generality. At this juncture, we note a classical result of Frobenius that addresses the special case of $m =1$ (see \cite{Frob,Dieudonne}), which deals with our simple running example earlier. This result is essential for our proof of the above theorem for any value of $m$. We will also give our own proof of this result as it allows us to illustrate our proof strategy in the simple case. 

\begin{theorem} [Frobenius] \label{theo:frob}
Let $S = \SING_{n,1} \subseteq V = \Mat_{n}$. The group of symmetries $\mathcal{G}_S$ consists of linear transformations of the form $X \mapsto PXQ$ or of the form $X \mapsto PX^tQ$ where $P,Q \in \SL_n$. 
\end{theorem}

First, note that the above result computes the entire group of symmetries! In the general case, let us first note that apriori there could be an incredible number of groups whose identity component is $G_{n,m}$. However, it turns out that they are actually manageable, and with some fairly elementary results on semisimple Lie algebras, we can determine the entire group of symmetries for any $m$. 

\begin{theorem} \label{theo:gos}
Let $S = \SING_{n,m} \subseteq V =  \Mat_{n}^m$. Let $\tau$ denote the linear transformation that sends $X = (X_1,\dots,X_m) \mapsto (X_1^t,\dots,X_m^t)$. Then the group of symmetries $\mathcal{G}_S = G_{n,m} \cup G_{n,m} \cdot \tau = G_{n,m} \rtimes \Z/2$.
\end{theorem}

The key idea here is that the entire group of symmetries must normalize the connected group of symmetries, i.e., $G_{n,m}$. So, we compute the normalizer of $G_{n,m}$. To do so, we utilize heavily that the group $G_{n,m}$ is reductive, and use ad-hoc arguments that are particularly suited to this special case. A slightly more abstract approach via automorphisms of Dynkin diagrams such as the one in \cite{Guralnick} would work in this case (see also \cite{Lbook17}). We do not quite know a general strategy to bridge the gap between the connected group of symmetries and the entire group of symmetries. We also note that the same strategy yields the group of symmetries for $\NSING_{n,m}$

\begin{theorem} \label{theo:ngos}
Let $S = \NSING_{n,m} \subseteq \Mat_{n}$. Then the group of symmetries $\mathcal{G}_S =  G_{n,m} \rtimes \Z/2$ (as defined in the above theorem).
\end{theorem}

Once we compute the group of symmetries, the rest of the argument relies on an understanding of the Hilbert--Mumford criterion (see Theorem~\ref{theo:HM}) which tells us that the null cone is a union of $G$-orbits of coordinate subspaces (linear subspaces that are defined by the vanishing of a subset of coordinates, see Definition~\ref{D-coordsubspace}). In particular, we will show that the union of all the coordinate subspaces contained in $\SING_{n,m}$ moved around by the action of its group of symmetries does not cover all of $\SING_{n,m}$, which will give the contradiction. We explain this idea in more detail in Section~\ref{subs:proofidea}.

\begin{remark} [Positive characteristic]
Our choice in working with $\C$ as a ground field is essentially for simplicity of the exposition and proofs. All our results above (specifically Theorems~\ref{theo:nullcone}, \ref{theo:nullcone2}, \ref{theo:cgos}, \ref{theo:frob}, \ref{theo:gos} and \ref{theo:ngos}) hold for every algebraically closed  fields of every characteristic. In Appendix~\ref{app.pos.char}, we discuss the issues that arise in positive characteristic and the appropriate modifications needed to deal with them.
\end{remark}

The subvariety $\SING_{n,m}$ is  the zero locus of some very structured polynomials. Observe that for any $c_i \in \C$, the polynomial $\det(\sum_i c_iX_i)$  vanishes on $\SING_{n,m}$. It is easy to see that the zero locus of the collection of all $\det(\sum_i c_i X_i)$ (for all choices of $c_i$) is precisely $\SING_{n,m}$\footnote{It seems plausible that these polynomials generate the ideal of polynomials that vanish on $\SING_{n,m}$, but such questions can often be quite subtle to prove.}. We prove a negative result for Problem~\ref{prob.invring} for this collection of polynomials.

\begin{theorem} \label{theo:invring}
Suppose $n,m \geq 3$. Then the subring $R = \C[\{\det(\sum_i c_i X_i)  : c_i \in \C\}] \subseteq \C[\Mat_{n}^m]$ is not the invariant ring for any linear action of any group $G$ on $\Mat_{n}^m$.
\end{theorem}

If we restrict to reductive groups, then the above theorem is a simple consequence of Theorem~\ref{theo:nullcone} and the alternate definition of null cone as the zero locus of non-constant homogenous invariants. However, we use a different argument that works for {\em any} group, irrespective of reductivity.

\subsection{Organization}
In Section~\ref{sec:invthry}, we recall the basic notions from invariant theory and null cones as well as the crucial Hilbert--Mumford criterion. It also contains a sketch of the proof strategy for proving Theorem~\ref{theo:nullcone}. In Section~\ref{sec:gos}, we present the theoretical statements that we will use in the computation of the group of symmetries, and in particular, we describe the role of Lie algebras. This is followed by Section~\ref{sec:explicit}, which contains an explicit description of the action of the Lie algebra on polynomials, which is vital for our computations. The ideal of polynomials vanishing on $\SING_{n,m}$ is discussed in Section~\ref{sec:van}, with proofs pushed into the appendix. The group of symmetries for $\SING_{n,1}$ (i.e., the important special case of $m = 1$) is computed in Section~\ref{sec:frob}. While the statement was already known (due to Frobenius), we present a different proof that serves to illustrate our strategy in the general case. Section~\ref{sec:multi} is a discussion of a particular multi-grading of the set of matrix-tuples, needed for computations. Section~\ref{sec:intermediate} tackles an intermediate problem, for a simpler group action. The group of symmetries for $\SING_{n,m}$ and $\NSING_{n,m}$   are computed in Section~\ref{sec:symsingnm}, proving Theorem~\ref{theo:gos} and Theorem~\ref{theo:ngos}. Section~\ref{sec:notnullcone} contains the proofs of Theorem~\ref{theo:nullcone}, our main negative result, and Theorem~\ref{theo:nullcone2} which implies it. In Section~\ref{sec:inv.conv}, we prove Theorem~\ref{theo:invring}. Finally, in Section~\ref{sec:disc}, we discuss some open problems and directions for future research. 

In Appendix~\ref{App.gos}, we recall the necessary algebraic geometry and Lie theory to prove the results in Section~\ref{sec:gos}. The results stated in Section~\ref{sec:van} are proved in Appendix~\ref{App.rep} with the help of representation theory. Finally in Appendix~\ref{app.pos.char}, we discuss the modifications needed to extend the results to positive characteristic.

\subsection*{Acknowledgements}
We would like to especially thank J. M. Landsberg for suggesting that we compute the group of symmetries, and Gurbir Dhillon for helping us with the Lie theoretic statements needed for the computation. In addition, we also thank Ronno Das, Harm Derksen, Ankit Garg, Robert Guralnick, Alexander Kleschev, Thomas Lam, Daniel Litt, Rafael Oliveira, Gopal Prasad, Akash Sengupta, Rahul Singh, Yuval Wigderson, John Wiltshire-Gordon and Jakub Witaszek for helpful discussions.

\section{Invariant theory and null cones} \label{sec:invthry}
We will now recall the basic notions in invariant theory that we need. In particular, we will need the notions of rational group actions (rational representations), their invariant polynomials, null cones and their basic properties. Most of this material is well known and can be found in a standard text such as \cite{DK}. We will try to remain as elementary as possible. We remind the reader again that our underlying field is $\C$. 

A {\em linear algebraic group} $G$ is a subgroup of $\GL_n$ (for some $n$) that is also an algebraic subvariety\footnote{An equivalent definition is that a linear algebraic group $G$ is an (affine) algebraic variety $G$ which is also a group such that the multiplication map $m: G \times G \rightarrow G$ and an inverse map $i:G \rightarrow G$ are morphisms of algebraic varieties. While this definition seems more general, it is a standard result that both definitions agree. For this reason, sometimes linear algebraic groups are also called affine algebraic groups.}. In this paper, we will drop the prefix linear and simply refer to these as algebraic groups for brevity. The connected component of an algebraic group $G$ containing the identity element is itself a connected algebraic group, and we call this the identity component of $G$, and denote it by $G^\circ$.

For a linear algebraic group $G \subseteq \GL_n$, an $m$-dimensional representation is simply a map $\rho:G \rightarrow \GL_m$. We want to consider ``algebraic representations'', so we want the map $\rho$ to be a morphism of algebraic varieties. So, for $X = (x_{ij}) \in G \subseteq \GL_n$, each coordinate of the $m \times m$ matrix $\rho(X) \in \GL_m$ is given as a rational function (ratio of polynomials) in the $x_{ij}$'s.\footnote{Note that morphism $\rho$ needs to be a regular morphism (and not a rational morphism) as it must be defined on all of $G$. In particular even though $\rho(X)$ is given by a matrix of rational functions, all these rational functions have to be defined on $G \subseteq \GL_n$, so their locus of indeterminacy must be away from $G$. A canonical example is the function $\frac{1}{\det}$ which is an honest ratio of polynomials that will be defined on $G$ (and indeed all of $\GL_n$). Also observe that $\frac{1}{\det}$ is a regular function on $\GL_n$ and hence on $G$ as well.} This is why such representations are called {\em rational} representations. The definition itself is of course quite straightforward.

\begin{definition} [Rational representation]
A rational representation $V$ of an algebraic group $G$ is a morphism of algebraic groups $G \rightarrow \GL(V)$ (where $V$ is a vector space over $\C$). By a morphism of algebraic groups, we simply mean a group homomorphism that is also a morphism of varieties. 
\end{definition}

A morphism $G \rightarrow \GL(V)$ can also be interpreted as a morphism $G \times V \rightarrow V$, and we will write $g \cdot v$ or simply $gv$ to denote the image of $(g,v)$ under this map. The orbit of a point $v \in V$ is $G \cdot v = \{gv\ |\ g \in G\}$. All representations considered in this paper will be rational. Subrepresentations, direct sums etc are defined in the standard way. A representation is called {\em irreducible} if it has no subrepresentations. 

\begin{remark}
For $V$ to be a rational representation of an algebraic group $G$ simply means that $G$ acts algebraically on $V$ by linear transformations. This is precisely the premise under which we define a null cone, and hence precisely the hypothesis in the main results (for example in Theorem~\ref{theo:nullcone}).
\end{remark}

For a vector space $V$, we denote by $\C[V]$ the ring of polynomial functions on $V$ (a.k.a. the coordinate ring of $V$). Concretely, if we have a basis $e_1,\dots,e_n$ for $V$, and $x_1,\dots,x_n$ denote the corresponding coordinate functions, then $\C[V] = \C[x_1,\dots,x_n]$ is the polynomial ring in $\dim V = n$ variables.

\begin{definition} [Invariant function]
For a representation $V$ of a group $G$, a function $f \in \C[V]$ is invariant (for the action of $G$) if it is constant along orbits, i.e., $f(gv) = f(v)$ for all $v \in V$ and $g \in G$.
\end{definition}

Invariant functions form a subring of the coordinate ring, which we will call the {\em invariant ring} or {\em ring of invariants}. 

\begin{definition} [Invariant ring]
For a representation $V$ of a group $G$, we denote by $\C[V]^G$, the ring of invariants, i.e.,
$$
\C[V]^G = \{f \in \C[V]\ |\ f(gv) = f(v)\ \forall g\in G,v\in V\}.
$$
\end{definition}

Invariant rings are graded subrings of the polynomial ring $\C[V]$, i.e., $\C[V]^G = \bigoplus_{d \in \N} \C[V]^G_d$.

There are several equivalent definitions of a reductive group, particularly in characteristic zero. We pick a definition that would resonate with anyone who has had experience with representations of finite groups. In particular, we want to point to the fundamental result called Maschke's theorem, which says that for a finite group $G$ (if characteristic is zero or doesn't divide $|G|$), any representation can be written as a direct sum of irreducible representations (a.k.a. complete reducibility). This property is very useful because in order to study any representation of $G$, one can often reduce it to the study of the irreducible representations. Algebraic groups with this property are called reductive groups.

\begin{definition} [Reductive group]
An algebraic group $G$ is called reductive if any rational representation $V$ of $G$ is completely reducible, i.e., it can be written as a direct sum of irreducible representations. 
\end{definition}

Examples of reductive groups include all finite groups, tori (i.e., $(\C^*)^n)$, and all classical groups such as $\GL_n$, $\SL_n$, ${\rm SO}_n$, ${\rm Sp}_n$ etc.

\begin{definition} [Null cone] \label{def:nullcone}
Let $V$ be a rational representation of a reductive group $G$. Then the null cone $\mathcal{N}_G(V)$ (or simply $\mathcal{N}_G$ or even $\mathcal{N}$ when there is no confusion) is defined as the set of points in $V$ whose orbit closure\footnote{The orbit closure can be taken in the Zariski topology or the analytic topology, since they are both the same.} contains zero, i.e., 
$$\mathcal{N} = \mathcal{N}_G(V) = \{v \in V\ |\ 0 \in \overline{G \cdot v}\},
$$
where $\overline{G \cdot v}$ denotes the closure of $G \cdot v$, the orbit of $v$.
\end{definition}

The above definition of the null cone is analytic in nature, and as defined seems to be a feature of the geometry of orbits and their closures. However, there is an equivalent algebraic description via invariant polynomials that we state below due to Mumford (and known already to Hilbert for $G = \SL_n$). This interplay between the analytic and algebraic viewpoints has already proved extremely valuable (see for e.g. \cite{GGOW16,BGFOWW,BGFOWW2}), and perhaps yet to be exploited to its full capacity.

\begin{theorem} [Mumford] \label{theo:alg.nullcone}
Let $V$ be a rational representation of a reductive group $G$. Then, the null cone 
$$
\mathcal{N} = \mathcal{N}_G(V) = \text{ zero locus of } \bigcup\limits_{d= 1}^{\infty} \C[V]^G_d.
$$
\end{theorem}

For a proof of the above result, we refer the reader to \cite[Section~2.5]{DK}.

\subsection{Null cones for tori} \label{sec:inv.thry.tori}
The group $\C^* = \GL_1(\C)$ is clearly an algebraic group, which is moreover abelian. A direct product $T = (\C^*)^n$ is called a (complex) torus. Any connected abelian reductive group must a torus! Needless to say (non-abelian) reductive groups can of course be far more complicated than tori. However, an understanding of the null cone for tori is key to understanding the null cones for more general reductive groups, and this is captured by the celebrated Hilbert--Mumford criterion that we will discuss in the next subsection. The null cone in the ``easy'' abelian case of the torus has a simple description as a union of linear subspaces of a specific form (this is related to the linear programming problem in complexity).

For this subsection, let $T = (\C^*)^n$ be a (complex) torus. Let $\mathcal{X}(T)$ denote all the characters of $T$, i.e., all algebraic group homomorphisms $T \rightarrow \C^*$. One can identify $\mathcal{X}(T) = \Z^n$ as follows. For $\lambda = (\lambda_1,\dots,\lambda_n) \in \Z^n$, we have the corresponding character (also denoted $\lambda$ by abuse of notation) $\lambda: T \rightarrow \C^*$ defined by $\lambda(t_1,\dots,t_n) = \prod_{i=1}^n t_i^{\lambda_i}$. It is a well known result that these are all the algebraic characters of $T$.

Suppose $V$ is a rational representation of $T$. Then there is a {\em weight space decomposition} 
$$
V = \bigoplus_{\lambda \in \mathcal{X}(T)} V_\lambda,
$$
where for any $\lambda \in \mathcal{X}(T)$, the weight space $V_{\lambda} = \{v \in V \ |\ t \cdot v = \lambda(t) v\}$. One should think of this as a simultaneous eigenspace decomposition for the action of $T$. Indeed, the weight space $V_\lambda$ consists of eigenvectors for the action of every $t \in T$, although each $t$ will act by a different eigenvalue, i.e., $\lambda(t)$. Elements of $V_{\lambda}$ are called weight vectors of weight $\lambda$. Let $e_1,\dots,e_m$ denote a basis of $V$ consisting of weight vectors (thus identifying $V$ with $\C^m$), and let the weight of $e_i$ be $w_i \in \Z^n$. Let $x_1,\dots,x_m$ denote the corresponding coordinates.

\begin{remark} \label{R-tori.invpol}
A monomial $\prod_{i=1}^m x_i^{a_i}$ is an invariant monomial if and only if $\sum_i a_i \cdot w_i = 0$ (note that $a_i \in \N$ and $w_i \in \Z^n$, so this is an equality in $\Z^n$). As a vector space over $\C$, the ring of invariants $\C[V]^T$ is spanned by such invariant monomials. In particular, the null cone is precisely the zero locus of such invariant monomials (excluding the trivial monomial $\prod_{i = 1}^m x_i^0$ which is the constant function $1$).
\end{remark}

\begin{definition} [Coordinate subspace] \label{D-coordsubspace}
For a subset $I \subseteq [m]$, we define $L_I$ to be the linear subspace of $\C^m$ that is defined as the zero locus of $\{x_j: j \notin I\}$. In other words, $L_I$ consists of all the vectors in $\C^m$ whose support (i.e., the set of non-zero coordinates) is a subset of $I$. We will call any subspace of the form $L_I$ a {\em coordinate subspace}.
\end{definition}

For a subset $I \subseteq [m]$, consider the set of points $W_I = \{w_i : i \in I\} \subseteq \Z^n \subseteq \Q^n \subseteq \R^n$. Let $\Delta_I$ denote the convex hull of $W_I$. The following description of the null cone is the main takeaway from this subsection. We provide a proof for completeness.

\begin{proposition}
Let $V$ be an $m$-dimensional rational representation of the torus $T$. Identify $V = \C^m$ using a basis of weight vectors. Using the notation above, the null cone
$$
\mathcal{N}_T(V) = \bigcup\limits_{I \subseteq [m], 0 \notin \Delta_I} L_I
$$
\end{proposition}

\begin{proof}
First, let us show that for each $I$ such that $0 \notin \Delta_I$, $L_I \subseteq \mathcal{N}_T(V)$. By Remark~\ref{R-tori.invpol} and Theorem~\ref{theo:alg.nullcone}, it suffices to show that every (non-constant) invariant monomial vanishes on $L_I$. Take such an invariant monomial $m = \prod_{i=1}^m x_i^{a_i}$. If $a_j > 0$ for some $j \notin I$, then clearly $m$ vanishes on $L_I$. Otherwise $m = \prod_{i \in I} x_i^{a_i}$, so for $m$ to be invariant, $\sum_{i \in I} a_i w_i = 0$, but this means that $0 \in \Delta_I$, which is a contradiction. Thus every non-constant invariant monomial vanishes on $L_I$. Thus, we have shown $\supseteq$.

For the reverse direction, it suffices to show that $v \notin {\rm R.H.S.}$ implies $v \notin \mathcal{N}_T(V)$. To this end, let $v \notin {\rm R.H.S.}$. Let $J$ be the support of $v$ (i.e., the set of all non-zero coordinates). Clearly $0 \in \Delta_J$. Thus, we have $0 = \sum_{i \in J} a_i w_i = 0$ for some $a_i > 0$ and $\sum_i a_i = 1$. If the $a_i$'s were (non-negative) integers, then $\prod_{i \in J} x_i^{a_i}$ would be an invariant monomial that doesn't vanish at $v$, and we would be done. Even if the $a_i$'s are rational numbers, by removing denominators, the argument would still go through. However, we only know that $a_i$'s are real numbers, and we will need a little bit of work to ensure that we can get a non-negative integer linear combination of the $w_i$'s to add to zero. 

W.l.o.g., we can assume that $a_i > 0$ for all $i \in J$ (else, replace $J$ with $\{i \in J\ | a_i >0\}$ and proceed with the argument). Let $K = \{(p_i)_{i \in J}\ | \ \sum_{i \in J} p_i w_i = 0\} \subseteq \R^J$. Then $K$ is the kernel of an $n \times |J|$ matrix whose columns are $w_i : i \in J$. Since this matrix has rational entries, there is a basis of $K$ with rational entries, i.e., $b_1,\dots,b_r \in K \cap \Q^J$ that span $K$ (as an $\R$-vector space). Now, since $(a_i)_{i \in J} \in K$, we can write $(a_i)_{i \in J} = \sum_{t = 1}^r \lambda_t b_t$ for some $\lambda_t \in \R$. Since $\sum_t \lambda_t b_t = (a_i)_{i \in J}  \in \R_{> 0}^J$, we deduce by continuity that there exists $\epsilon > 0$ such that $\sum_t \mu_t b_t \in \R_{>0}^J$ for all $\mu_i$ such that $|\mu_i - \lambda_i| < \epsilon$. Let $\mu_i$ be rational numbers such that $|\mu_i - \lambda_i| < \epsilon$. Then let $(c_i)_{i \in J} = \sum_i \mu_i b_i \in K \cap \R_{>0}^J$, but clearly $c_i$ are rational, so $(c_i)_{i \in J} \in K \cap \Q_{>0}^J$. Thus $\sum_i c_i w_i = 0$ and $c_i \in \Q$. For some $D \in \N$, we have $Dc_j \in \N$ for all $j$. Then $m = \prod_{i \in J} x_i^{Dc_j}$ is an invariant monomial that does not vanish on $v$, so $v \notin \mathcal{N}_T(V)$.
\end{proof}

\subsection{Null cones for reductive groups: Hilbert--Mumford criterion}
Let $G$ be a reductive group, and let $T$ be a maximal torus\footnote{All maximal tori are conjugate. Moreover, the union of all maximal tori is dense in the identity component of $G$.}, i.e., a subgroup of $G$ that is a torus, and not contained in a larger torus. The celebrated result called the Hilbert--Mumford
criterion that we state below essentially tells us that elements in the null cone for the action of $G$ are precisely those which can be moved (by applying an element of $G$) into the null cone for the torus $T$. In particular, this is one way to see that the null cone for the group $G$ is the same as the null cone for its identity component $G^\circ$.

The following statement can be found in \cite{Mumford} (see also \cite[Theorem~2.5.3]{DK}).

\begin{theorem} [Hilbert--Mumford criterion] \label{theo:HM}
Let $V$ be a rational representation of a reductive group $G$. Then 
$$
\mathcal{N}_G(V) = G \cdot \mathcal{N}_T(V).
$$
\end{theorem}

The Hilbert--Mumford criterion is sometimes stated in more general fashion, which says that $v \in V$ is in the null cone for $G$ if and only if there is a $1$-parameter subgroup of $G$ that drives it to zero. To see that this is equivalent to the version we state above, one needs to understand two things. The first is that any $1$-parameter subgroup is contained in some maximal torus, and all maximal tori are conjugate. The second is an understanding of the $1$-parameter subgroups of a torus, which will show that our description of the null cone for tori agrees with the criterion in terms of $1$-parameter subgroups (this is not hard).

\subsection{Proof idea of Theorem~\ref{theo:nullcone2}} \label{subs:proofidea}
Let us briefly give the idea behind the proof of Theorem~\ref{theo:nullcone2} (from which Theorem~\ref{theo:nullcone} follows easily). Indeed, suppose there is a reductive group $G$ (with maximal torus $T$) acting on $V = \Mat_{n}^m$ preserving $\SING_{n,m}$ such that the null cone is contained in $\SING_{n,m}$. Then the null cone for the torus $N_T(V)$ is also a subset of $\SING_{n,m}$. We know from the above discussion that $\mathcal{N}_T(V)$ is a union of coordinate subspaces. We will show that any coordinate subspace contained in $\SING_{n,m}$ must already be contained in $\NSING_{n,m}$ -- we will see this in Section~\ref{sec:notnullcone}. Thus, whatever $\mathcal{N}_T(V)$ may be, it must be contained in $\NSING_{n,m}$. 

This is the point where an understanding the group of symmetries is really needed. To be precise, the crucial result that drives the following argument is that the group of symmetries for $\SING_{n,m}$ is the {\em same} as the group of symmetries for $\NSING_{n,m}$! So, in particular, since $G$ preserves $\SING_{n,m}$, it also preserves $\NSING_{n,m}$. Thus, we have $G \cdot \mathcal{N}_T(V) \subseteq \NSING_{n,m}$. By the Hilbert--Mumford criterion (Theorem~\ref{theo:HM} above), we get that $\mathcal{N}_G(V) = G \cdot \mathcal{N}_T(V) \subseteq \NSING_{n,m}$ which is the required conclusion for Theorem~\ref{theo:nullcone2}.

\section{Computing the group of symmetries via polynomials} \label{sec:gos}
In this section, we will explain the important statements that go into the calculation of the group of symmetries. The proofs will be postponed to an appendix so as to not interrupt the flow of the paper. The main purpose of this section is however to highlight the fact one can determine the connected group of symmetries by a linear algebraic computation (by passing to Lie algebras), and this works in great generality. Later on, we discuss a technique to determine the entire group of symmetries, but this works only in a more limited setting (which of course includes $\SING_{n,m}$).

Given a subset $S \subseteq V$, we define 
$$
I_S = \{f \in \C[V]\ |\ f(s) =0\  \forall s \in S\},
$$
which is called the ideal of polynomials vanishing on $S$. about $S$ can almost always be reformulated in terms of questions on $I_S$. The first observation is that for a group $G$ acting (algebraically) on $V$, there is a (natural) induced action of $G$ on $\C[V]$. To understand this action , we need to describe for $g \in G$ and a polynomial function $f \in \C[V]$, what the resulting polynomial function $g \cdot f$ is. To describe a polynomial function, one can simply give its evaluation on all points of $V$. The polynomial function $g \cdot f$ is defined by
$$
(g \cdot f) (v)  = f(g^{-1}v) 
$$

There are other ways to describe this action, one of them being that we identify $\C[V]$ with the symmetric algebra over the dual space $V^*$. This point of view is not needed here, but will be helpful in a later technical section. For now, we note some key features. The most important feature is that $\deg(f) = \deg(g \cdot f)$ for any $g \in G$ and any (homogenous) $f \in \C[V]$. Hence, the linear subspace $\C[V]_a$ consisting of homogenous polynomials of degree $a$ is a $G$-stable subspace of $\C[V]$. 

%Our principal use of this will be that $\C[V]_{\leq d}$ and $\C[V]_d$ are both finite dimensional as opposed to $\C[V]$ which is infinite dimensional.

When $S$ is a cone (i.e., $\lambda \in \C, s\in S \implies \lambda s \in S$), then $I_S$ is graded, i.e., $I_S = \oplus_{a \in \N} (I_S)_a$, where $(I_S)_a$ denotes the polynomials in $I$ that are homogenous of degree $a$.

\begin{lemma} \label{Lgos-graded}
Suppose $S \subseteq V$ is a cone, and let $a \in \N$. Then 
$$
\mathcal{G}_S \subseteq \{g \in \GL(V)\ |\ g (I_S)_a \subseteq (I_S)_a\}.
$$
Further, if the zero locus of $(I_S)_a$ is equal to $S$, then we have equality. 
\end{lemma}

The key reason behind restricting ourselves to polynomials of a certain degree is to work with finite dimensional vector spaces rather than infinite dimensional ones.

%\begin{proof}
%For $g \in \mathcal{G}_S$, we know that $g I_S \subseteq I_S$. So, for $f \in (I_S)_a$, we have $g \cdot f \in I_S$. Further $g \cdot f$ is a homogenous polynomial of degree $a$ in $I_S$, i.e, $g \cdot f \in (I_S)_a$. 

%Let us now consider the situation when the zero locus of $(I_S)_a$ is equal to $S$. Let $g \in \GL(V)$ be such that $g \cdot (I_S)_a \subseteq (I_S)_a$. By Lemma~\ref{L:incltoeq}, we see that $g \cdot (I_S)_a = (I_S)_a$ (and hence $g^{-1} (I_S)_a = (I_S)_a$). Let $s \in S$. We want to show that $g \cdot s \in S$. For any $f \in (I_S)_a$, we have $f(gs) = (g^{-1}\cdot f) (s) = 0$ since $g^{-1} \cdot f \in (I_S)_a$, and $s \in S$. Thus $gS \subseteq S$. 

%Now, let $v \in S^c$ (the complement of $S$ in $V$). There exists $f \in (I_S)_a$ such that $f(v) \neq 0$. Thus, $(g\cdot f) (gv) = f(v) \neq 0$. Hence $gv \notin S$ because $g \cdot f \in (I_S)_a$, and the zero locus of $(I_S)_a$ is precisely $S$. Thus $gS^c \subseteq S^c$. Since $g$ acts by an invertible transformation, $gS \subseteq S$ and $gS^c \subseteq S^c$ implies that $gS = S$, i.e., $g \in \mathcal{G}_S$.
%\end{proof}

Algebraic sets (in particular cones), are often described as the zero locus of a collection of (homogenous) polynomials $\{f_i: i \in I\}$. While it is difficult to compute a set of generators for the ideal $I_S$, the degrees of $\{f_i: i \in I\}$ can help us find a suitable $a$ to apply the above lemma (indeed the least common multiple of degrees of $f_i$ will suffice, but in specific cases, one can probably do much better). It is however another task to compute {\em all} the homogenous polynomials of a certain degree that vanish on $S$. In any specific case, this may be manageable, but we do not know of any general strategy. For the case of $\SING_{n,m}$, we will manage this (in a later section) with the help of representation theory of $\GL_m \times \GL_n \times \GL_n$.

%Finally, let us note that despite our characterization of the group of symmetries in Lemma~\ref{L-gos-truncated} and Lemma~\ref{Lgos-graded}, it still remains difficult to compute. However, there is measurable progress. Since $\C[V]_{\leq d}$ is a representation of $\GL(V)$, we have a morphism of algebraic groups $\rho: \GL(V) \rightarrow \GL(\C[V]_{\leq d})$. Lemma~\ref{L-gos-truncated} simply says that the group of symmetries of $S$ is the intersection of $\rho(\GL(V))$ with the group of symmetries of $(I_S)_{\leq d}$. The latter group of symmetries is easy to describe and has an upper block triangular form in a well chosen basis, but it is not quite so easy to get a handle on its intersection with $\rho(\GL(V))$. A similar analysis holds for Lemma~\ref{Lgos-graded}. 

Our technique to compute the group of symmetries $\mathcal{G}_S$ has two parts to it. The first is to determine the connected group of symmetries $\mathcal{G}_S^\circ$, i.e., the identity component of the group of symmetries -- this will be done by appealing to the theory of Lie algebras, which reduces the problem to linear algebra. The second is to determine the component group $\mathcal{G}_S/\mathcal{G}_S^\circ$, which is always a finite group.

\subsection{Connected group of symmetries via Lie algebras} 
In this section, we discuss the first part of our technique, i.e., how to determine the connected group of symmetries. The first observation (and easy to see) is that the group of symmetries of an algebraic subset $S \subseteq V$ is a Zariski closed subgroup of $\GL(V)$, and so is an algebraic subgroup of $\GL(V)$ (and hence a Lie group). Consequently, the connected group of symmetries is a connected algebraic group (and hence connected Lie subgroup). Connected Lie subgroups of $\GL(V)$ are in $1-1$ correspondence with Lie subalgebras of $\gl(V)$, the Lie algebra of $\GL(V)$\footnote{This is a famous result due to Chevalley, and is often called the subgroups-subalgebras correspondence}. Thus, to determine the connected group of symmetries $\mathcal{G}_S^\circ$, it suffices to determine its Lie algebra (denoted $\g_S$), which we will call the {\em Lie algebra of symmetries}. We should point out here that in general neither $\{M \in \gl(V)\ |\ M \cdot S = S\}$ nor $\{M \in \gl(V)\ |\ M \cdot S \subseteq S\}$ is equal to $\g_S$. However, we have the following result:

\begin{proposition} \label{P-Liealg-all}
Let $S \subseteq V$ be a cone. Then for any $a \in \N$, the Lie algebra of symmetries
$$
\g_S \subseteq \{M \in \gl(V)\ |\ M \cdot (I_S)_a \subseteq (I_S)_a\}.
$$
Finally, if the zero locus of $(I_S)_a$ is precisely the cone $S$, then we have equality.
\end{proposition}

In the appendix, we give a gentle and quick introduction to Lie algebras, and prove the proposition. It is however imperative for the reader to understand the action of the Lie algebra $\gl(V)$ on polynomials in $\C[V]_a$ to be able to use the above proposition as a computational tool. For this purpose, in the next section, we describe this action explicitly.

\begin{notation}
 For algebraic groups $G,H,\dots$, we will denote their Lie algebras by $\Lie(G),\Lie(H),\dots$ or by the corresponding gothic letters $\g,\h,\dots$ (to avoid cumbersome notation).
\end{notation}

%To compute $\g_S$, we will have to adopt the algebraic-geometric perspective\footnote{This is precisely why we have taken $S$ to be an algebraic subset.}.

%Let us first understand the action of the Lie algebra $\Lie(\GL(V)) = \gl(V)$ on $\C[V]$, the ring of polynomial functions, which we will now discuss. 

\subsection{Component group}
We have already discussed above how to compute the connected group of symmetries $\mathcal{G}_S^\circ$ for an algebraic subset $S \subseteq V$. To compute the entire group of symmetries $\mathcal{G}_S$, we observe that $\mathcal{G}_S$ is an algebraic subgroup of $\GL(V)$ whose identity component is $\mathcal{G}_S^\circ$. One deduces that $\mathcal{G}_S$ must be a subgroup of the normalizer of $\mathcal{G}_S^\circ$. In the event that $\mathcal{G}_S^\circ$ is a reductive group and acts irreducibly on $V$, its normalizer will be a finite extension (see for e.g. \cite{Guralnick})\footnote{One way to compute the normalizer in this case is by understanding the automorphisms of the Dynkin diagram, see \cite{Guralnick}. However, we will give more concrete arguments as many of our readers may not possess an in depth knowledge of the theory of semisimple algebras.}. In particular its normalizer is also an algebraic group whose identity component is $\mathcal{G}_S^\circ$. So, $\mathcal{G}_S$ is a union of some of the components of the normalizer of $\mathcal{G}_S^\circ$, and we just have to identity which ones.

%The next subsection will be devoted to discussing this issue.

%Using the aforementioned techniques, we compute the Lie algebra of symmetries for $\SING_{n,m}$. We compute

%\begin{proposition} \label{prop:laos}
%Let $S = \SING_{n,m} \subseteq V = \Mat_{n,n}^m$. Then $\g_S = \Lie(G_{\rm n,m})$.
%\end{proposition}

%We note here that the above proposition implies Theorem~\ref{theo:cgos} because $G_{\rm n,m}$ is a connected Lie group whose Lie algebra is $\Lie(G_{\rm n,m})$ and we know that connected Lie subgroups of $\GL(V)$ are in $1-1$ correspondence with Lie subalgebras of $\gl(V)$.

\section{Explicit description of the Lie algebra action on polynomials} \label{sec:explicit}
For the technical aspects of the computations we do, it is absolutely essential to understand the action of the Lie algebras on polynomial functions. Let $V$ be a vector space with basis $e_1,\dots,e_n$, and let the corresponding coordinates functions be denoted $x_1,\dots,x_n$. Using the basis, we identify $V = \C^n$, $\C[V] = \C[x_1,\dots,x_n]$, $\GL(V) = \GL_n$  and $\gl(V) = \Mat_{n}$. 

Let $E_{ij}$ denote the elementary matrix with a $1$ in $(i,j)^{th}$ entry and $0$'s everywhere else. As a linear transformation, $E_{ij}$ maps $e_j$ to $e_i$ and kills $e_k$ for $k \neq j$. The matrices $\{E_{ij}\}_{1 \leq i,j \leq n}$ form a basis for $\Mat_{n}$. We will describe the action of $E_{ij}$'s, and then extend by linearity to understand the action of $\gl(V)$. The matrix $E_{ij}$ acts as the derivation $-x_j \partial_i$, where $\partial_i$ denotes the partial derivative with respect to $x_i$. In other words, for any $f \in \C[V]$, we have
$$
E_{ij} \cdot f = -x_j\partial_i f
$$

To write it all out explicitly, a matrix $M \in \gl(V) = \Mat_{n}$ acts on a polynomial $f \in \C[V]  = \C[x_1,\dots,x_n]$ by the following formula:

$$
M \cdot f = \left(- \sum_{1\leq i,j \leq n} m_{ij} x_j \partial_i \right) f,
$$

where $m_{ij}$ denotes the $(i,j)^{th}$ entry of the matrix $M$. Also note that we can write $M = \sum_{i,j} m_{ij} E_{ij}$, so another point of view is that $m_{ij}$ are the coordinates of the matrix $M$ with respect to the (standard) basis $\{E_{ij}\}$.

\subsection{Twisting by Cartan involution}
The action of the Lie algebra on polynomial functions is annoying due to the negative signs and these will be cumbersome to keep track of it in computation. To make the computations less confusing and more intuitive, we twist the action. This is done with the help of the Cartan involution.

\begin{definition} [Cartan involution]
The Cartan involution $\Theta: \gl_n \rightarrow \gl_n$ is the composition of negation and transpose, i.e., 
$$
\Theta(X) = - X^t.
$$
The Cartan involution is an automorphism of Lie algebras.
\end{definition}

The main thing to observe about the map $\Theta$ is that it is an involution, i.e., $\Theta \circ \Theta$ is the identity map. 

\begin{definition} [Twisted action]
For $M \in \gl_n$, we define an action of $M$ on $\C[V] = \C[x_1,\dots,x_n]$ by 
$$
M \cd f  = \Theta(M) \cdot f.
$$
In particular, we have
$$
E_{ij} \cd f = x_i \partial_j f,
$$
and hence
$$
M \cd f = (\sum_{i,j} m_{ij} x_i \partial_j ) f
$$
\end{definition}

The result on computing the Lie algebra of symmetries (i.e., Proposition~\ref{P-Liealg-all}) can be reformulated in terms of the twisted action:

\begin{lemma} \label{star-action}
Let $S \subseteq V$ be a cone. Then for any $a \in \N$, we have
$$
\g_S \subseteq \{\Theta(M)\ |\ M \cd (I_S)_a \subseteq (I_S)_a \}  = \Theta \{M\ |\ M \cd (I_S)_a \subseteq (I_S)_a \}.
$$
In the above, we have equality if the zero locus of $(I_S)_a$ is precisely $S$.
\end{lemma}

\section{Vanishing ideal of singular tuples of matrices} \label{sec:van}
In order to apply the ideas of Section~\ref{sec:gos} to computing the Lie algebra of symmetries for $S$, we need to understand the ideal $I_S$. We will focus on the case of $S = \SING_{n,m}$. For this, one needs the representation theory of the general linear group (highest weight vectors, Cauchy formulas, Schur functors) as well as an understanding of the invariants in the left-right action of $\SL_n \times \SL_n$. The proofs are ad-hoc, and suited precisely to the case of $\SING_{n,m}$. It is highly unlikely that these ideas can be generalized to give results for other choices of $S$. For all these reasons, we postpone the proofs to an appendix. 

In the case of $S = \SING_{n,m}$, we do not know how to determine the entire ideal $I_S$. However, we can determine it upto degree $n$, which turns out to suffice for our purposes.

\begin{proposition} \label{P-ideal}
Let $S = \SING_{n,m} \subseteq V = \Mat_{n}^m$, and let $I_S \subseteq \C[\Mat_{n}^m]$ be the ideal of polynomial functions that vanish on $S$. Then,
\begin{enumerate}
\item $I_S$ is graded;
\item $(I_S)_a$ is empty if $a < n$;
\item $(I_S)_n = \spa (\det(\sum_i c_iX_i): c_i \in \C).$
\end{enumerate}
\end{proposition}

% For the second part, for any homogeneous polynomial function of degree $a$ (for $a < n$), we show an explicit point in $\SING_{n,m}$ on which it doesn't vanish. For the last part, we will need the representation theory of $\GL_m \times \GL_n \times \GL_n$, and in particular the theory of highest weight vectors. We will not actually need the full power of explicitly constructing highest weight vectors, but rather just the fact that highest weight vectors capture most of the information about the representations.

The significance of the above result is demonstrated by the following statement:

\begin{corollary} \label{C-loscom}
Let $S = \SING_{n,m} \subseteq V = \Mat_{n}^m$. Then
$$
\mathcal{G}_S = \{g \in \GL(V) \ |\ g \cdot (I_S)_n \subseteq (I_S)_n \}, 
$$
and hence
$$
\g_S = \{M \in \gl(V)\ | M \cdot (I_S)_n \subseteq (I_S)_n \}.
$$
\end{corollary}

\begin{proof}
This follows from Lemma~\ref{Lgos-graded}, Proposition~\ref{P-Liealg-all} and the above proposition since the zero locus of $(I_S)_n$ is precisely $\SING_{n,m}$
\end{proof}

The latter part of the corollary is the one that is extremely useful because now the computation of $\g_S$ is feasible. In principle, for a fixed $n$ and $m$, one could run an explicit computer algorithm to compute $\g_S$. However, we will actually be able to compute $\g_S$ for all $n,m$, and for this, we will need to do the linear algebra by hand. To do so, we will (repeatedly) exploit the numerous symmetries and multilinearity of the determinant polynomial.

\section{Symmetries of singular matrices} \label{sec:frob}
In this section we give a proof of Frobenius' result, i.e., Theorem~\ref{theo:frob} which is the important case of $m=1$ in Theorem~\ref{theo:gos}. The technique we use is different from the existing proofs in \cite{Frob,Dieudonne}. First, we use the the previous sections to write the computation of the Lie algebra of symmetries as a linear algebraic computation. Next, we define Kronecker product of matrices, and then describe the twisted action of $\gl(\Mat_{n})$. Then, we recall a few facts on the Symmetric group, and finally give the explicit computations to determine the Lie algebra of symmetries, which suffices to determine the connected group of symmetries. Finally, we compute the entire group of symmetries.

For this section, let $S = \SING_{n,1} = \{X \in \Mat_{n}\ | \det(X) = 0 \} \subseteq V = \Mat_{n}$. For $(P,Q) \in \GL_n \times \GL_n$, consider the linear transformation $X \mapsto PXQ^t$. The set of all such linear transformations is the group $G_{n,1}$ (as defined in Theorem~\ref{theo:gos}). We will write $\g_{n,1}$ for the Lie algebra of $G_{n,1}$. In this section, we will first prove:

%\footnote{Note that if we took $(P,Q) \in \sll_n \times \sll_n$ instead, then we would obtain $G_{n,1}$.}. We will denote by $\g_{n,1}$, the Lie algebra of $G_{n,1}$.

\begin{proposition} \label{P1}
Let $S = \SING_{n,1} \subseteq V = \Mat_{n}$. Then, we have
$\mathcal{G}_S^\circ = G_{n,1}
$ and hence
$\g_S = \g_{n,1}.
$
\end{proposition}

%For $(M,N) \in \Mat_{n} \times \Mat_{n}$, consider the linear transformation $X \mapsto MX + XN$. The set of all such linear transformations is $\g_{n,1}$, a subalgebra of $\gl(V)$.

Let $E_{ij}$ denote the $n \times n$ matrix with a $1$ in its $(i,j)^{th}$ spot and $0$'s everywhere else. Then $\{E_{ij}\}_{1 \leq i,j \leq n}$ form a basis for $V = \Mat_{n}$. The Lie algebra $\gl(V) = \Mat_{n^2}$ can be identified canonically with $\End(V)$, the space of linear transformations from $V$ to $V$. %The Lie bracket is given by the commutator (but this is not important for our purposes). 

Let us give a description of $\g_{n,1}$. We have
$$
\g_{n,1} = \{ X \mapsto MX + XN\ |\ M,N \in \Mat_{n} \} \subseteq \gl(V) = \End(V)\}.
$$

The Lie algebra of symmetries consists of precisely those elements of $\gl(V)$ for whose action the determinant polynomial is an eigenvector. 

\begin{lemma} \label{L-los-frob}
The Lie algebra of symmetries
$$
\g_S = \Theta(\{M \in \gl(V) \ |\ M \cd \det = c \cdot \det \text{ for some } c\in \C\})
$$
\end{lemma}

\begin{proof}
This follows from Proposition~\ref{P-ideal}, Corollary~\ref{C-loscom} and Lemma~\ref{star-action}.
\end{proof}

To make the necessary computations, we need to understand the action (or rather the twisted action) of $\gl(\Mat_{n})$ on polynomials in $\C[\Mat_{n}]$ explicitly. Prior to that, we recall the notion of Kronecker product of matrices.

\subsection{Kronecker product of matrices}
For $A = (a_{ij}) \in \Mat_{n}$ and $B = (b_{ij}) \in \Mat_{n}$, we define the Kronecker (or tensor) product 
$$
A \otimes B = 
\begin{pmatrix}
a_{11}B & a_{12}B & \dots & a_{1n}B \\
a_{21}B & \ddots & \ddots & \vdots \\
\vdots & \ddots & \ddots & \vdots \\
a_{n1}B & \dots & \dots & a_{nn} B
\end{pmatrix} \in \Mat_{n^2}.
$$

Using the Kronecker product, we will identify $\Mat_{n^2} = \Mat_{n} \otimes \Mat_{n}$. To do so, we will first index the rows and columns by $[n]\times [n]$ in lexicographic order (rather than $[n^2]$). Thus a basis for $\Mat_{n^2}$ is given by $\{E_{ij,kl}\ |\ ij,kl \in [n] \times [n]\}$. Note that we have the intuitive equality
$$
E_{ij} \otimes E_{pq} = E_{ip,jq}.
$$

\subsection{Twisted action of $\gl(\Mat_{n})$ on $\C[\Mat_{n}]$}
Recall that the (standard) basis for $V = \Mat_{n}$ is $\{E_{ij}\}$, and let $\{x_{ij}\}$ denote the corresponding coordinate functions. Note that $V$ is $n^2$-dimensional. The Lie algebra $\gl(V) = \gl(\Mat_{n})$ can be identified with $\Mat_{n^2}$, to be viewed as $\Mat_{n} \otimes \Mat_{n}$, as described above. Observe that we can think of an $n^2 \times n^2$ matrix as an $n \times n$ block matrix whose blocks are also of size $n \times n$. Thus for any $M \in \gl(V)$, we have

$$
M = \begin{pmatrix} 
M_{11} & \dots & M_{1n} \\
\vdots & \ddots & \vdots \\
M_{n1} & \dots & M_{nn}
\end{pmatrix},
$$
where each $M_{ij}$ is an $n\times n$ matrix. Note that writing $M = \sum_{i,j} E_{ij} \otimes M_{ij}$ also defines $M_{ij}$. 

The matrix $E_{ij} \otimes E_{pq} = E_{ip,jq} \in \Mat_{n^2} = \gl(\Mat_{n})$ acts (twisted action) on polynomial functions via the derivation $x_{ip} \partial_{jq}$, i.e., for $f \in \C[\Mat_{n}]$, we have
\begin{equation}
E_{ip,jq} \cd f = x_{ip}\partial_{jq} f.
\end{equation}

There is an $\N^n$-grading on $\C[V]  = \C[x_{ij}]$ given by setting $\deg(x_{ij}) = \delta_i :=  (0,\dots,0,\underbrace{1}_i,0,\dots,0) \in \N^n$. We have the decomposition
$$
\C[V] = \bigoplus_{d \in \N^n} \C[V]_d,
$$
where $\C[V]_d$ denotes the (multi)-homogeneous polynomials of degree $d$. The polynomial $\det = \sum_{\sigma \in S_n} \prod_{i=1}^n x_{i \sigma(i)}$ has degree $(1,1,\dots,1)$. We make a crucial observation

\begin{remark}
Let us understand the action of the Lie algebra with respect to this multi-degree. The twisted action of the matrix $E_{ip,jq}$ is by the derivation $x_{ip}\partial_{jq}$. Thus, $E_{ip,jq} : \C[V]_d \rightarrow \C[V]_{d - \delta_j + \delta_i}$.
\end{remark}

\begin{definition} [Grading on $\gl(\Mat_{n})$]
We give a grading on the Lie algebra $\gl(\Mat_{n})$ by setting $\deg(E_{ip,jq}) = \delta_i - \delta_j$. We have 
$$
\gl(\Mat_{n}) =  \Mat_{n^2} = \gl(\Mat_{n})_0 \bigoplus_{i \neq j}\gl(\Mat_{n})_{\delta_i -\delta_j}.
$$

For any $M \in \Mat_{n^2}$, we write $M = \sum_{i,j} E_{ij} \otimes M_{ij}$ for matrices $M_{ij} \in \Mat_{n}$ as above. Then the degree $0$ part is $\sum_i E_{ii} \otimes M_{ii}$, and for $i \neq j$, $E_{ij} \otimes M_{ij}$ is the degree $\delta_i - \delta_j$ part. Thus, the decomposition of $M$ into homogeneous components is
$$
M = (\sum_i E_{ii} \otimes M_{ii}) \bigoplus_{i \neq j} E_{ij} \otimes M_{ij}
$$

\end{definition}

The following lemma is immediate from the preceding discussion.

\begin{lemma}
If $M \in \Mat_{n^2}$ is homogeneous of degree $d$, and $f \in \C[V]$ is homogeneous of degree $d'$, then $M \cd f$ is homogeneous of degree $d+ d'$
\end{lemma}

The following easy corollary of the above lemma will be useful to us:

\begin{corollary}
Suppose $M = \sum_{ij} E_{ij} \otimes M_{ij} \in \gl(V)$.  Then $M \cd \det = c \cdot \det$ for some $c \in \C$ if and only if the following conditions hold:
\begin{enumerate}
\item For $i \neq j$, $(E_{ij} \otimes M_{ij}) \cd \det = 0$;
\item $\left(\sum_{i=1}^n E_{ii} \otimes M_{ii}\right) \cd \det = c \cdot \det$.
\end{enumerate}
\end{corollary}

\begin{proof}
This is straightforward from the decomposition of $M$ into homogenous components (given above) and the above lemma.
\end{proof}

Before we unravel the above condition to compute the Lie algebra of the symmetries, a few words on the symmetric group.

\subsection{Symmetric group}
 We denote by $S_n$, the symmetric group on $n$ letters. In other words, $S_n$ consists of all bijective maps $\sigma: [n] \rightarrow [n]$. The group operation is composition of maps. The pair $(i,j)$ with $i < j$ is called an inversion for $\sigma \in S_n$ if $\sigma(i) > \sigma(j)$. For $\sigma \in S_n$, we define it sign 
$$
\sgn(\sigma) = (-1)^{\text{ number of inversions in $\sigma$ }}.
$$

For $\sigma \in S_n$, we will define $\iota(\sigma) \in S_n$ by $\iota(\sigma)(1) = \sigma(2)$, $\iota(\sigma(2)) = \sigma(1)$ and $\iota(\sigma)(k) = \sigma(k)$ for all $k > 2$. Thus
$$
\iota:S_n \rightarrow S_n
$$ 
is an involution (without any fixed points!). Moreover, $\sgn(\iota(\sigma)) = - \sgn(\sigma)$ for any $\sigma \in S_n$.

\subsection{Computation of Lie algebra of symmetries}
%Having developed the necessary technical details for the Lie algebra computations necessary, let us now compute $\g_S$ in this case. Combining Lemma~\ref{lsiffdet} and the above corollary we get:

%\begin{lemma}
%Let $M = \sum_{ij} E_{ij} \otimes M_{ij} \in \gl(V) = \gl(\Mat_{n}) = \Mat_{n^2}$. Then $M \in \g_S$ if and only if 
%\begin{enumerate}
%\item For $i \neq j$, $(E_{ij} \otimes M_{ij}) \cd \det = 0$;
%\item $\left(\sum_{i=1}^n E_{ii} \otimes M_{ii}\right) \cd \det = c \cdot \det$.
%\end{enumerate}
%\end{lemma}

To understand $\g_S$ (i.e., the elements of $\gl(V)$ for which $\det$ is an eigenvector by Lemma~\ref{L-los-frob}) it suffices to understand the two conditions in the previous corollary. For the rest of this subsection, let $M = \sum_{ij} E_{ij} \otimes M_{ij} \in \Mat_{n^2} = \gl(V)$.

\begin{lemma} \label{L1-offd}
For $i \neq j$, $(E_{ij} \otimes M_{ij}) \cd \det = 0$ if and only if $M_{ij} = \kappa {\rm I}_n$ for some $\kappa \in \C$, where ${\rm I}_n$ denotes the identity matrix  of size $n \times n$. 
\end{lemma}

\begin{proof}
 Let us first prove the forward direction. Suppose $E_{ij} \otimes M_{ij} \cd \det = 0$. We want to prove that $M_{ij} = \kappa {\rm I}_n$.

First, observe that without loss of generality, we can consider $i = 2$ and $j = 1$, and we will do so. Let us denote the $(p,q)^th$ entry of $M_{21}$ by $\alpha_{pq}$. Then $E_{21} \otimes M_{21}$ acts by the derivation $D = \sum_{p,q} \alpha_{p,q} x_{2p} \partial_{1q}$.

Recall that $$
\det = \sum_{\sigma \in S_n} \sgn(\sigma) x_{1\sigma(1)}x_{2\sigma(2)}\dots x_{n \sigma(n)}.
$$

Thus
\begin{align*}
(E_{21} \otimes M_{21}) \cd \det & = D \cdot \det \\
& = \sum_{\begin{array}{c} p,q,\sigma \\ \sigma(1)  = q\end{array}} \alpha_{pq} \sgn(\sigma) x_{2p}x_{2\sigma(2)}x_{3\sigma(3)}\dots x_{n\sigma(n)}.
\end{align*}

% In the derivation $D$, consider the term $\alpha_{pq} x_{2p} \partial_{1q}$ for some $p \neq q$. Also consider a permutation sigma such that $\sigma(1) = q$. Then 

%$$
%(\alpha_{pq} x_{2p} \partial_{1q}) \cdot x_{1\sigma(1)}\dots x_{n\sigma(n)} = \alpha_{pq} x_{2p} x_{2\sigma(2)}\dots x_{n\sigma(n)}.
%$$

%Now, we claim that when we expand $D \cdot \det$, let us consider the coefficient of the monomial $x_{2p} x_{2\sigma(2)}\dots x_{n\sigma(n)}.$

%Suppose for some $\pi \in S_n$, and some $r,s$, this monomial occurs in $\alpha_{rs} x_{2r} \partial_{1s} \prod_{i=1}^n x_{i \pi(i)}.$. This means that 

Consider the monomial $m = x_{22}x_{22}x_{33}\dots x_{nn}$. Let us compute the coefficient of $m$ in $D \cdot \det$. To do so, let us check the choices of $p,q$ and $\sigma$ in the above summation contribute to the coefficient of $m$. Surely, we need $2 = p = \sigma(2)$, and $\sigma(k) = k$ for $k > 2$. This means that $\sigma$ must be the identity permutation, $p = 2$ and $q = \sigma(1) = 1$. So, there is only one contributing term, and that contributes a coefficient of $\alpha_{21}$. Since $D \cdot \det = 0$, we must have that $\alpha_{21} = 0$. 

For any choice of $p \neq q$, a similar argument will show that $\alpha_{pq} = 0$ for $p \neq q$ (indeed consider instead of $m$, a monomial $x_{2q} x_{2q} x_{3\pi(3)}\dots x_{n\pi(n)}$ for some $\pi \in S_n$ such that $\pi(1) = p$ and $\pi(2) = q$). This proves that the off-diagonal entries of $M_{ij}$ are zero.

Now, consider the monomial $n = x_{21}x_{22}x_{33}\dots x_{nn}$ and let us compute its coefficient in $D \cdot \det$. Again, let us check the choices of $p,q$ and $\sigma$. Clearly need that $\sigma(k) = k$ for $k \geq 2$. Moreover, we need either $p = 1$ and $\sigma(2) = 2$ or $p = 2$ and $\sigma(2) = 1$. In the former case, we will have $\sigma$ to be the identity permutation and $q = \sigma(1) = 1$, so this contributes $\alpha_{11}$. Similarly the latter case contributes $- \alpha_{22}$.  Thus, the coefficient of the $n$ is $\alpha_{11} - \alpha_{22}$ which must be zero. Hence $\alpha_{11}  = \alpha_{22}$.

Again, a similar argument proves that $\alpha_{ii} = \alpha_{jj}$ for all $i,j$ (indeed, consider instead of $n$, the monomial $x_{2i}x_{2j}x_{3\pi(3)}\dots x_{n\pi(n)}$ for some $\pi \in S_n$ such that $\pi(1) = i$ and $\pi(2) = j$). Thus, $M_{ij} = \kappa {\rm I}_n$, where we take $\kappa = \alpha_{11}$. This shows that if $E_{ij} \otimes M_{ij} \cd \det = 0$, then $M_{ij} = \kappa {\rm I}_n$. 

For the converse direction, if $M_{ij} = \kappa {\rm I}_n$, then $E_{ij} \otimes M_{ij}$ acts by 
$$
D = \kappa \sum_{i=1}^n x_{2i} \partial_{1i}.
$$
%One can simply check explicitly that $D \cdot \det = 0$, for example, by checking that the coefficient of all possible monomials vanishes.
%\comm{Do I have to write this out? It follows from a straightforward sign reversing involution of terms.}

Consider the action of $D$ on $t = \prod_{i=1}^n x_{i\sigma(i)}$. Unless $i = \sigma(1)$, the term $x_{2i} \partial_{1i}$ kills it. Thus, we get:
\begin{align*}
D \cdot t & = \kappa x_{2\sigma(1)} \partial_{1\sigma(1)} \cdot t \\
&= x_{2\sigma(1)} x_{2\sigma(2)} x_{3\sigma(3)} \dots x_{n \sigma(n)}. \\
\end{align*}

Thus, 
$$
D \cdot \det = \kappa \sum_{\sigma \in S_n} \sgn(\sigma) x_{2\sigma(1)} \partial_{1\sigma(1)} \cdot t = \kappa(\sum_{\sigma \in S_n} \sgn(\sigma) x_{2\sigma(1)} x_{2\sigma(2)} x_{3\sigma(3)} \dots x_{n \sigma(n)}).
$$

We claim that the sum is zero. To see this, notice that the terms corresponding to $\sigma$ and $\iota(\sigma)$ cancel. Hence, the whole sum cancels out as required.
\end{proof}

\begin{lemma} \label{L1-diag}
Suppose $\left(\sum_{i=1}^n E_{ii} \otimes M_{ii}\right) \cd \det = c \cdot \det$. Then for all $i,j$, we have $M_{ii} - M_{jj} = \mu_{i,j} {\rm I}_n$ for some scalar $\mu_{i,j} \in \C$.
\end{lemma}

\begin{proof}
Let the $(p,q)^{th}$ entry of $M_{ii}$ be $\alpha^{i}_{p,q}$. Without loss of generality, take $i = 1$ and $j=2$, i.e., we will prove $M_{11} - M_{22} = \mu_{1,2} {\rm I}_n$ for some $\mu_{1,2} \in \C$.
Note that $\sum_i E_{ii} \otimes M_{ii}$ acts by the derivation $D = \sum_{i,p,q} \alpha^{i}_{p,q} x_{ip} \partial_{iq}$. So

\begin{equation} \label{Stupidity}
D \cdot \det = \sum_{\begin{array}{c} i,p,q,\sigma \\ \sigma(i) = q \end{array}} \alpha_{p,q}^i \cdot \sgn(\sigma) \cdot x_{1\sigma(1)}\dots x_{ip} \dots x_{n\sigma(n)}.
\end{equation}

Fix $\pi \in S_n$, and consider the monomial $m = x_{1\pi(1)} x_{2\pi(1)} x_{3\pi(3)} \dots x_{n\pi(n)}$. The coefficient of $m$ in $c \cdot \det$ is $0$. So, the coefficient of $m$ in $D \cdot \det$ is also zero. We leave it to the reader to check from the above expression that the coefficient of $m$ in $D \cdot \det$ is $\sgn(\pi) (\alpha_{\pi(1),\pi(2)}^2 - \alpha_{\pi(1),\pi(2)}^1)$. Thus, we have $\alpha_{\pi(1),\pi(2)}^2 = \alpha_{\pi(1),\pi(2)}^1$. Running over all choices of $\pi$, we get that $\alpha_{p,q}^1 = \alpha_{p,q}^2$ for all $p \neq q$. This means that the off-diagonal entries of $M_{11}$ and $M_{22}$ are the same.

Again fix $\pi \in S_n$. Consider the monomial $n = \prod_{i=1}^n x_{i\pi(i)}$. Its coefficient in $c \cdot \det$ is $c \cdot \sgn(pi)$. So, its coefficient in $D \cdot \det$ should also be $c \cdot \sgn(\pi)$. From Equation~\ref{Stupidity}, one can check again that the coefficient of $n$ in $D \cdot \det$ is $\sgn(\sigma) (\sum_i \alpha_{\pi(i),\pi(i)}^i).$ Thus, we must have
$$
\sum_i \alpha^i_{\pi(i),\pi(i)} = c.
$$
This holds for all permutations, in particular, if we replace $\pi$ with $\iota(\pi)$. Thus, we have

$$
\sum_i \alpha^i_{\pi(i),\pi(i)} = \sum_i \alpha^i_{(\iota\pi)(i),(\iota\pi)(i)}.
$$

Hence, this means

$$
\alpha^1_{\pi(1),\pi(1)} + \alpha^2_{\pi(2),\pi(2)} = \alpha^1_{\pi(2),\pi(2)} + \alpha^2_{\pi(1),\pi(1)}.
$$

Again, this holds for all $\pi$, so we have

$$
\alpha^1_{p,p} - \alpha^2_{p,p} = \alpha^1_{q,q} - \alpha^2_{q,q}.
$$
for all $p,q$. Thus, if we set $\mu_{1,2} = \alpha^1_{p,p} - \alpha^2_{p,p}$, then the diagonal entries of $M_{11}$ and $M_{22}$ differ precisely by $\mu_{1,2}$. Since the off-diagonal entries of $M_{11}$ and $M_{22}$ agree (shown above), we have that $M_{11} - M_{22} = \mu_{1,2} {\rm I}_n$.

\end{proof}

\begin{corollary}
Suppose $M \in \Mat_{n^2} = \gl(\Mat_{n})$. Then $M \cd \det = c \cdot \det$ for some $c\in \C$ if and only if $M$ is of the form $A \otimes {\rm I}_n + {\rm I}_n \otimes B$. 
\end{corollary}

\begin{proof}
This follows from the previous two lemmas. Suppose $M$ is such that $M \cd \det = c \cdot \det$ for some $c \in \C$. Observe that $M_{ij} = \kappa_{ij} {\rm I}_n$ by Lemma~\ref{L1-offd} for some scalars $\kappa_{ij} \in \C$. Then, by Lemma~\ref{L1-diag}, we know that $M_{jj} = M_{11} + \mu_{j,1} {\rm I}_n$ for all $j$ (where $\mu_{j,1}$ is as defined in Lemma~\ref{L1-diag}).

Now, set $A$ to be the $n \times n$ matrix whose $(i,j)^{th}$ entry is $\kappa_{ij}$ if $i \neq j$, and $(i,i)^{th}$ entry is $\mu_{j,1}$. (Note that $\mu_{1,1} = 0$). Also, set $B = M_{11}$. This just means that $M = A \otimes {\rm I}_n + {\rm I}_n \otimes B$.
\end{proof}

\begin{corollary}
The Lie algebra of symmetries $\g_S = \Theta \{A \otimes {\rm I}_n + {\rm I}_n \otimes B\ | A,B \in \Mat_{n}\} = \{A \otimes {\rm I}_n + {\rm I}_n \otimes B\ | A,B \in \Mat_{n}\}$.
\end{corollary}

\begin{proof}
It follows from the above corollary and Lemma~\ref{L-los-frob} that $\g_S = \Theta \{A \otimes {\rm I}_n + {\rm I}_n \otimes B\ | A,B \in \Mat_{n}\}$. Observe that $\Theta( A \otimes {\rm I}_n + {\rm I}_n \otimes B) = -(A^t \otimes {\rm I}_n + {\rm I}_n \otimes B^t)$. Thus $\Theta \{A \otimes {\rm I}_n + {\rm I}_n \otimes B\ | A,B \in \Mat_{n}\} = \{A \otimes {\rm I}_n + {\rm I}_n \otimes B\ | A,B \in \Mat_{n}\}$.

\end{proof}

\begin{lemma}
The lie algebra $\g_{n,1} = \{ A \otimes {\rm I}_n + {\rm I}_n \otimes B\ | A,B \in \Mat_{n}\} \subseteq \Mat_{n^2} = \gl(\Mat_{n})$.
\end{lemma}

\begin{proof}
We have the map $\phi: \GL_n \times \GL_n \rightarrow \GL_{n^2}$ given by $(P,Q) \mapsto P \otimes Q$. The image of $\phi$ is $G_{n,1}$ (by definition). The derivative of $\phi$, i.e., $d\phi: \gl_n \times \gl_n \rightarrow \gl_{n^2} = \gl(\Mat_{n})$ is given by the formula $d\phi(A,B) = A \otimes {\rm I}_n + {\rm I}_n \otimes (-B)$. The lemma follows from the fact that the image of $d\phi$ is the Lie algebra of the image of $\phi$ (see Appendix~\ref{App.gos}), i.e., $\Lie(G_{n,1}) = \g_{n,1}$. 
\end{proof}

From the above lemma and the corollary preceding it, we deduce Proposition~\ref{P1}.

\begin{proof} [Proof of Proposition~\ref{P1}]
It follows from the above lemma and the preceding corollary that $\g_S = \g_{n,1}$. Hence, it follows that $G_S^\circ = G_{n,1}$.
\end{proof}

\subsection{The entire group of symmetries}
Much of the work has gone into computing the connected group of symmetries $\mathcal{G}_S^\circ$. Now, we want to determine the entire group of symmetries $\mathcal{G}_S$. To do so, we will need some results from the theory of semisimple Lie algebras, and we will only recall those facts that we need.

\begin{lemma}
Let $G$ be a linear algebraic group and let $G^\circ$ denote its identity component. Then $G$ normalizes $G^\circ$, i.e., for all $g \in G$, we have $g G^\circ g^{-1} = G^\circ$.
\end{lemma}

\begin{proof}
Consider the map $\phi:G^\circ \rightarrow G$ given by $h \mapsto ghg^{-1}$. The image is a connected because $G^\circ$ is connected, and contains the identity $e \in G$ because $\phi(e) = e$. So, we must have $g G^\circ g^{-1} \subseteq G^\circ$. Since this holds for any $g \in G$, we must have $g^{-1} G^\circ g \subseteq G^\circ$, which implies that $G^\circ \subseteq g G^\circ g^{-1}$. Thus, we have $g G^\circ g^{-1} = G^\circ$.
\end{proof}

\begin{definition}
Let $H \subseteq G$ be a subgroup. The normalizer of $H$ in $G$ is defined as 
$$
N_G(H) = \{g \in G \ |\ gHg^{-1} = H\}.
$$
\end{definition}

Let $F_n \subseteq \GL(\Mat_{n}) = \GL_{n^2}$ denote the subgroup of all linear transformations of the form $X \mapsto PXQ$ and $X \mapsto PX^tQ$ for some $P,Q \in \GL_n$. The proof of the following lemma is from \cite{Dokovic-Li}, but we recall it as we will need to generalize it. 

\begin{lemma} \label{L-normfrob}
The normalizer of $H = G_{n,1}$ in $G = \GL(\Mat_{n})$ is $F_n$.
\end{lemma}

\begin{proof}
Let $\tau:\Mat_{n} \rightarrow \Mat_{n}$ denote the transpose, i.e., $\tau(A) = A^t$. Clearly $\tau \in N_G(H)$, so $F_n \subseteq N_G(H)$.

Let $g \in N_G(H)$. Then $g$ normalizes $H$, and hence normalizes its derived group $[H,H] = G_{n,1}$. Therefore, it normalizes its Lie algebra $\g_{n,1} = \Lie(G_{n,1}) = \{A \otimes {\rm I}_n + {\rm I}_n \otimes B\ |\ {\rm Tr}(A) = {\rm Tr}(B) = 0\}$. Abstractly, $\g_{n,1}$ is equal to $\sll_n \oplus \sll_n$. Let $L_1 = \{A \otimes {\rm I}_n \ |\  {\rm Tr}(A) = 0\}$ and let $L_2 = \{{\rm I}_n \otimes B\ |\ {\rm Tr}(B) = 0\}$. Then $\g_{n,1} = L_1 \oplus L_2$ is a semisimple Lie algebra, and each $L_i$ is a simple Lie algebra isomorphic to $\sll_n$. Thus if $g$ normalizes $\g_{n,1}$, then conjugation by $g$ is an automorphism of the Lie algebra. $L_1$ and $L_2$ are the only simple ideals of $\g_{n,1}$, so such an automorphism must either preserve each $L_i$ or switch the two.
Also, observe that conjugation by $\tau$ switches $L_1$ and $L_2$. Thus, by composing with $\tau$ if necessary, we assume w.l.o.g that conjugation by $g$ preserves each $L_i$.

Now, write $g = \sum_{i=1}^r P_i \otimes Q_i$ with $\{P_i\}$ a linearly independent subset of $\Mat_{n}$  and $\{Q_i\}$ also a linearly independent subset of $\Mat_{n}$. Since $gL_1g^{-1}  = L_1$ (equivalently $gL_1 = L_1 g$), we have that for any $A \in \Mat_{n}$ (with trace zero), there exists $\widetilde{A} \in \Mat_{n}$ (with trace zero) such that
$$
\sum_i P_i A \otimes Q_i = \sum_i \widetilde{A} P_i \otimes Q_i.
$$

Since $Q_i$ are linearly independent, we deduce that $P_i A = \widetilde{A} P_i$ for all $i$. Mutliplying by an appropriate $U \otimes {\rm I}_n$ on the left and $V \otimes {\rm I}_n$ on the right (both of which are in $H$), we can assume that $P_1 = \begin{pmatrix} {\rm I}_k & 0 \\ 0 &  0 \end{pmatrix}$ for some $k \leq n$. We claim that $k = n$. Otherwise, take $A$ to be $\begin{pmatrix} 0  & E \\ 0 & 0 \end{pmatrix}$ for some non-zero $E$, and observe that there is no $\widetilde{A}$ which can satisfy $P_1A = \widetilde{A} P_1$. Thus $k = n$, i.e., $P_1 = {\rm I}_n$. Hence $A = P_1 A = \widetilde{A} P_1 = \widetilde{A}$, i.e., $\widetilde{A} = A$ for all $A$. Thus we must have $P_i A = A P_i$ for all $i \geq 2$. This means that $P_i$ are scalar matrices, but we chose $P_i$ to be linearly independent. So, we must have $i = 1$. 

This means that $g = P_1 \otimes Q_1$, with $P_1$ invertible. A similar argument shows that $Q_1$ is also invertible. In other words, $g \in H$.

To summarize, we have that either $g$ or $g \tau$ is in $H$, i.e, $g \in F_n$. This shows that $N_G(H) \subseteq F_n$. Thus $N_G(H) = F_n$ as required.
\end{proof}

\begin{proof} [Proof of Theorem~\ref{theo:frob}]
We know that $\mathcal{G}_S$ is an algebraic subgroup of $\GL(\Mat_{n})$ whose identity component is $G_{n,1}$ by Proposition~\ref{P1}. Thus, we must have $\mathcal{G}_S \subseteq N_{\GL(\Mat_{n})}(G_{n,1}) = F_n$ by the above discussion. On the other hand, it is clear that $F_n \subseteq \mathcal{G}_S$, so $\mathcal{G}_S = F_n$ as required.
\end{proof}

\begin{remark}
A very short proof (that hides much of the details) is to observe that $F_n \subseteq \mathcal{G}_S \subsetneq \GL(\Mat_{n})$. Next, $F_n$ is a maximal proper subgroup of $\GL(\Mat_{n})$ (see for example \cite{Dynkin, Dokovic-Li, Guralnick}), so $\mathcal{G}_S = F_n$. However, in the general case, such an argument will not suffice, but the argument we give here will be generalized.
\end{remark}

\section{Multi-grading on $\C[\Mat_{n}^m]$} \label{sec:multi}
For any vector space $V$, the ring of polynomial functions $\C[V]$ has a natural grading given by (total) degree. In the case when $V = \Mat_{n}^m$, we have a finer multi-grading which we will now describe.

We will denote the coordinate functions on $\Mat_{n}^m$ by $x^{(i)}_{jk}$. More precisely, let $x^{(i)}_{jk}$ denote the $(j,k)^{th}$ coordinate of the $i^{th}$ matrix. Thus, 
$$
\C[\Mat_{n}^m] = \C[x^{(i)}_{jk} : 1 \leq i \leq m, 1\leq j,k \leq n].
$$
We can define an $\N^m$-grading on $\C[\Mat_{n}]$ by setting
$$
\deg(x^{(i)}_{jk}) = \delta_i = (0,\dots,0,\underbrace{1}_i,0 \dots 0) \in \N^m.
$$
So, we have
$$
\C[\Mat_{n}^m] = \bigoplus_{e = (e_1,\dots,e_m) \in \N^m} \C[\Mat_{n}^m]_{e}.
$$
where $\C[\Mat_{n}^m]_{e}$ is the linear span of all monomials $m = \prod (x^{(i)}_{jk})^{a^{(i)}_{jk}}$ such that $\deg(m) = \sum a^{(i)_{jk}} \delta_i  = e$.

We will now give another description of $(I_S)_n$ for the case $S = \SING_{n,m}$ that incorporates this multi-grading. First, recall that $(I_S)_n = \spa(\det(\sum_i c_i X_i): c_i \in \C)$. Let $t_1,\dots,t_m$ denote indeterminates. For $e \in \N^m$, let $t^e := t_1^{e_1} t_2^{e_2} \dots t_m^{e_m}$. Consider $\det(\sum_i t_i X_i) \in \C[\Mat_{n}]$. Write
$$
\det(\sum_i t_i X_i) = \sum_{e \in \N^m, \sum_i e_i = n} t^e f_e.
$$

\begin{lemma} \label{L-multi-I}
Let $S = \SING_{n,m} \subseteq V = \Mat_{n}$. Then we have
\begin{enumerate}
\item $(I_S)_n = \bigoplus\limits_{e \in \N^m, \sum_i e_i = n} (I_S)_e$ ;
\item $(I_S)_e = \spa(f_e)$, and hence $1$-dimensional.
\end{enumerate}
\end{lemma}

\begin{proof}
A standard interpolation argument tells us that $\spa(f_e: e \in \N^m \text{ such that } \sum_i e_i = n) = \spa(\det(\sum_i c_i X_i) : c_i \in \C)$. It is also easy to see that the multi-degree $\deg(f_e) = e$. The lemma now follows. 
\end{proof}

Let us describe more explicitly the polynomials $f_e$. We call $p = (p_1,\dots,p_n) \in [m]^n$ $e$-compatible if $|\{i \ |\  p_i = j\}| = e_j$ for $1 \leq j \leq m$.

We have
$$
f_e = \sum_{p \text{ is } e\text{-compatible}} \sum_{\sigma \in S_n} \sgn(\sigma) x^{(p_1)}_{1\sigma(1)} x^{(p_2)}_{2\sigma(2)}\dots x^{(p_n)}_{n\sigma(n)}.
$$

%It is illustrative to see this in some specific cases. For example, if $e = (n,0,\dots,0)$, then the only compatible $p$ is $(1,1,\dots,1)$. So, 

\section{An intermediate problem: action of $(\GL_n \times \GL_n)^{\times m}$} \label{sec:intermediate} 
In this section, we consider the action of $(\GL_n \times \GL_n)^{\times m}$ on $\Mat_{n}^m$ given by 
$$
((P_1,Q_1),\dots,(P_m,Q_m)) \cdot (X_1,\dots,X_m) = (P_1X_1Q_1^{t},\dots,P_mX_mQ_m^{t}).
$$

The main goal of this section is to compute the subgroup of $(\GL_n \times \GL_n)^{\times m}$ which fixes $\SING_{n,m}$. Our approach here will be slightly different because we do not resort to a Lie algebra computation. Consider the homomorphism $\GL_n \times \GL_n \times (\C^*)^m \rightarrow (\GL_n \times \GL_n)^{\times m}$ given by 
$$
(P,Q,(\lambda_1,\dots,\lambda_m),(\mu_1,\dots,\mu_m)) \mapsto ((\lambda_1P,\mu_1Q), (\lambda_2 P, \mu_2Q),\dots, (\lambda_m P, \mu_mQ)).
$$
Let the image of this homomorphism be denoted $H$.

\begin{proposition} \label{P-inter}
The subgroup of $(\GL_n \times \GL_n)^{\times m}$ that preserves $S = \SING_{n,m} \subseteq V = \Mat_{n}$ is $H$, i.e., 
$$
\{g \in (\GL_n \times \GL_n)^{\times m} \ |\ g S = S\} = H
$$
\end{proposition}

We will derive the above proposition from the following lemma.

\begin{lemma} \label{L-inter}
Let $S = \SING_{n,m} \subseteq V = \Mat_{n}$. The subgroup $\{ g \in (\GL_n \times \GL_n)^{\times m} \ |\ g (I_S)_n \subseteq (I_S)_n \} = H$.
\end{lemma}

We will prove Proposition~\ref{P-inter} assuming Lemma~\ref{L-inter} above, and then we will prove Lemma~\ref{L-inter}.

\begin{proof} [Proof of Proposition~\ref{P-inter}]
Let $\rho:(\GL_n \times \GL_n)^{\times m} \rightarrow \GL(\Mat_{n}^m)$ be the group homomorphism that defines the representation above. Then, observe that
$\{g \in (\GL_n \times \GL_n)^{\times m} \ |\ g S = S\}  = \rho^{-1} (\mathcal{G}_S)$.

Similarly, $\{ g \in (\GL_n \times \GL_n)^{\times m} \ |\ g (I_S)_n \subseteq (I_S)_n \} = \rho^{-1} (\mathcal{G}_S)$ follows from Lemma~\ref{Lgos-graded} because the zero locus of $(I_S)_n$ is precisely $S$.

Thus, we have an equality
$$
\{g \in (\GL_n \times \GL_n)^{\times m} \ |\ g \SING_{n,m} = \SING_{n,m}\} =\{ g \in (\GL_n \times \GL_n)^{\times m} \ |\ g (I_S)_n \subseteq (I_S)_n \} = \rho^{-1}(\mathcal{G}_S).
$$

Hence, Lemma~\ref{L-inter} implies Proposition~\ref{P-inter}
\end{proof}

Now, all that is left to prove is Lemma~\ref{L-inter}. 

\begin{proof} [Proof of Lemma~\ref{L-inter}]
Let $g \in (\GL_n \times \GL_n)^{\times m}$ be such that $g (I_S)_n \subseteq (I_S)_n$. Write $g = ((P_1,Q_1^{t}),\dots,(P_m,Q_m^{t}))$ (We put transposes on the $Q$'s for convenience). For any $f \in (I_S)_n$, we have
$$
g^{-1} \cdot f (X) = f( gX) = f( P_1X_1Q_1,\dots, P_m X_m Q_m) = c \cdot f(X_1, L_2X_2R_2,\dots, L_mX_mR_m), 
$$

where $L_i = P_1^{-1}P_i$ and $R_i = Q_i Q_1^{-1}$, and $c = \det(P_1^{-1}Q_1^{-1})$. The last equality follows because $(I_S)_n$ is spanned by $\det(\sum_i c_i X_i)$, and $ \det(\sum_i c_i P_1^{-1} X_i Q_1^{-1}) = \det(P_1^{-1}Q_1^{-1}) \det(\sum_i c_i X_i) $

Recall the multi-degree defined in the previous section. Observe that for any $g \in (\GL_n \times \GL_n)^{\times m}$ and any (multi)-homogenous polynomial $f$ of degree $e$, the polynomial $g^{-1} \cdot f$ is also (multi)-homogenous of degree $e$.

From now on, let
$$
f = f_{(n-1,1,0,\dots,0)} = \sum_{\sigma \in S_n} \sum_{1 \leq r \leq n} x^{(1)}_{1\sigma(1)} \dots x^{(1)}_{(r-1),\sigma(r-1)} x^{(2)}_{r \sigma(r)} x^{(1)}_{(r+1),\sigma(r+1)} \dots x^{(1)}_{n\sigma(n)}.
$$

Then we have 
$$g^{-1} \cdot f (X) = f(gX) = c f(X) \text{ for some } c \in \C^*.
$$
 This is because $g^{-1} \cdot f$ is a non-zero polynomial that is homogeneous of degree $(n-1,1,0,\dots,0)$ and must be in $(I_S)_n$, and so is in $(I_S)_{(n-1,1,0,\dots,0)}$ which is spanned by $f = f_{(n-1,1,0,\dots,0)}$ by Lemma~\ref{L-multi-I}.

\begin{claim} 
The matrix $L_2$ is diagonal.
\end{claim}

\begin{proof} [Proof of Claim]
Suppose $(L_2)_{ij} \neq 0$ for some $i \neq j$ (Here $(L_2)_{ij}$ denotes the $(i,j)^{th}$ entry of $L_2$). Let $p$ and $q$ be such that $(R_2)_{pq} \neq 0$. This means that $(L_2)_{ij} (R_2)_{pq} \neq 0$. We have 
\begin{align*}
g^{-1}\cdot f &= f(gX) \\
& = f(X_1,L_2X_2R_2,\dots,L_mX_mR_m) \\
& = \sum_{\sigma \in S_n} \sgn(\sigma) \sum_r x^{(1)}_{1\sigma(1)} \dots x^{(1)}_{(r-1),\sigma(r-1)} \left(\sum_{a,b} (L_2)_{ra} \cdot x^{(2)}_{ab} \cdot (R_2)_{b\sigma(r)} \right) x^{(1)}_{(r+1),\sigma(r+1)} \dots x^{(1)}_{n\sigma(n)}. \\
& = \sum_{\sigma,r,a,b} \sgn(\sigma) \cdot (L_2)_{ra} \cdot (R_2)_{b\sigma(r)} \cdot x^{(1)}_{1\sigma(1)} \dots x^{(1)}_{(r-1),\sigma(r-1)}  x^{(2)}_{ab}  x^{(1)}_{(r+1),\sigma(r+1)} \dots x^{(1)}_{n\sigma(n)}.  
\end{align*}

Let $\pi \in S_n$ be such that $\pi(i) = q$. Let us compute the coefficient of the monomial $m = x^{(1)}_{1\pi(1)} \dots x^{(1)}_{(i-1),\pi(i-1)} x^{(2)}_{jp} x^{(1)}_{(i+1),\pi(i+1)} \dots x^{(1)}_{n\pi(n)}$ in $g^{-1} \cdot f$. In the expansion of $g^{-1} \cdot f$ above, let us see for which choices of $\sigma,r,a,b$ do we get the monomial $m$. Indeed we must have $ r = i$, and so we must have $\sigma(k) = \pi(k)$ for all $k \neq i$, which forces $\sigma =\pi$. Further, we must have $a = j$ and $b = q$. Hence, we conclude that the monomial $m$ appears in the above expansion exactly once and with a coefficient of $\sgn(\pi) (L_2)_{ij} (R_2)_{pq}$ (which is nonzero as noted above). However, the coefficient of the monomial $m$ in $c \dot f_{(n-1,1,0,\dots,0)}$ is zero. This is a contradiction. Therefore $(L_2)_{ij} = 0$ for all $i \neq j$. This means that $(L_2)$ is a diagonal matrix. 
\end{proof}

By a similar argument, all the $(L_i)$'s and $(R_i)$'s are all diagonal matrices. Now, let us write out $g^{-1} \cdot f$ again. Since all the $L_i$'s and $R_i$'s are diagonal, we have
$$
g^{-1}\cdot f = \sum_{\sigma \in S_n}  \sum_r x^{(1)}_{1\sigma(1)} \dots x^{(1)}_{(r-1),\sigma(r-1)} \left( (L_2)_{rr} x^{(2)}_{r \sigma(r)} (R_2)_{\sigma(r),\sigma(r)} \right) x^{(1)}_{(r+1),\sigma(r+1)} \dots x^{(1)}_{n\sigma(n)}.
$$ 

\begin{claim}
The matrices $L_i$ and $R_i$ are scalar matrices.
\end{claim}
\begin{proof} [Proof of Claim]
The coefficient of $n = x^{(1)}_{1\sigma(1)} \dots x^{(1)}_{(r-1),\sigma(r-1)} x^{(2)}_{r \sigma(r)} x^{(1)}_{(r+1),\sigma(r+1)} \dots x^{(1)}_{n\sigma(n)}$ in $g^{-1} \cdot f$ is $\sgn(\sigma)(L_2)_{rr} (R_2)_{\sigma(r),\sigma(r)}$. The coefficient of $n$ in $c \cdot f$ is $\sgn(\sigma) \cdot c$. Thus, if we are to have $g^{-1} \cdot f = c \cdot f$, then we must have $(L_2)_{rr} (R_2)_{\sigma(r),\sigma(r)} = c$. This must hold for all choices of $r$ and $\sigma$, so we have $(L_2)_{ii} (R_2)_{jj} = c$ for all $i,j$. This means that $(L_2)_{ii} = (L_2)_{kk}$ for all $i,k$, i.e., $L_2$ is a scalar matrix, and so is $R_2$. Similarly all the $L_i$ and $R_i$ are scalar matrices.
\end{proof}

Since the $L_i$'s and $R_i$'s are scalar matrices, we can write $L_i = \lambda_i {\rm I}_n$ and $R_i = \mu_i {\rm I}_n$ for scalars $\lambda_i,\mu_i \in \C^*$. Thus we have $P_i = \lambda_i P_1$ and $Q_i = \mu_i  Q_1$ for $i \geq 2$. Thus, we have
$$
g = ((P_1,Q_1),(\lambda_2P_1,\mu_2Q_2),\dots,(\lambda_mP_1,\mu_mQ_1)) \in H.
$$

To summarize, we have shown that $\{ g \in (\GL_n \times \GL_n)^{\times m} \ |\ g (I_S)_n \subseteq (I_S)_n \} \subseteq H$. The other inclusion is clear.
\end{proof}

\section{Symmetries of singular tuples of matrices} \label{sec:symsingnm}
In this section, we will compute the group of symmetries for $S = \SING_{n,m}$, i.e., Theorem~\ref{theo:gos}. While the high-level strategy resembles that of Section~\ref{sec:frob} (which deals with the $m =1 $ case), we need to work a little harder in the computations. Below, we will recall the setup again for the convenience of the reader. Then, we will describe the action of $\gl(\Mat_{n}^m)$ on polynomials explicitly. We then present the explicit computation of the Lie algebra of symmetries. The main features that we utilize in the computation are the multi-grading (as defined in Section~\ref{sec:multi}), the description of the ideal $I_S$ in Lemma~\ref{L-multi-I} and the intermediate case resolved in the previous section. This is followed by the computation of the entire group of symmetries which parallels the computation in Section~\ref{sec:frob}. Finally we indicate how the same arguments also compute the group of symmetries for $\NSING_{n,m}$.

First, let us recall the setup again. We have $S = \SING_{n,m} \subseteq V = \Mat_{n}^m$, and we want to compute $\mathcal{G}_S,\mathcal{G}_S^\circ$ and $\g_S$. We will first focus on $\g_S$. From Corollary~\ref{C-loscom}, we have
$$
\g_S = \{M \in \gl(\Mat_{n}^m) \ |\ M \cdot (I_S)_n \subseteq (I_S)_n \}
$$
Moreover, we have by Proposition~\ref{P-ideal} that
$$
(I_S)_n = \spa(\det(\sum_i c_i X_i) : c_i \in \C).
$$
From Lemma~\ref{L-multi-I}, we know that $(I_S)_n$ is multi-graded, the explicit decomposition being $(I_S)_n = \bigoplus\limits_{e \in \N^m, \sum_i e_i = n} (I_S)_e$, where $(I_S)_e$ is $1$-dimensional and spanned by $f_e$ (as defined in Lemma~\ref{L-multi-I}). Moreover, we have an explicit formula for $f_e$, i.e.,
\begin{equation} \label{fe}
f_e = \sum_{p \text{ is } e\text{-compatible}} \sum_{\sigma \in S_n} \sgn(\sigma) x^{(p_1)}_{1\sigma(1)}x^{(p_2)}_{2\sigma(2)} \dots x^{(p_n)}_{n\sigma(n)},
\end{equation}
where we call $p = (p_1,\dots,p_n) \in [m]^n$ $e$-compatible if $|\{i \ |\  p_i = j\}| = e_j$ for $1 \leq j \leq m$. Of particular interest are the cases of $e = (n,0,\dots,0)$ and $e = (n-1,1,0,\dots,0)$. We have:
\begin{equation}
f_{(n,0,\dots,0)} = \sum_{\sigma \in S_n} \sgn(\sigma) x^{(1)}_{1\sigma(1)}\dots x^{(1)}_{n\sigma(n)} = \det(X_1) 
\end{equation}
Similarly, we have:
\begin{equation}
f_{(n-1,1,0,\dots,0)} = \sum_{j \in [n]} \sum_{\sigma \in S_n} \sgn(\sigma) x^{(1)}_{1\sigma(1)} \dots x^{(1)}_{j-1, \sigma(j-1)} x^{(2)}_{j\sigma(j)}x^{(1)}_{j+1,\sigma(j+1)} \cdots x^{(1)}_{n \sigma(n)}.
\end{equation}

In the next subsection, we will write out the (twisted) action of $\gl(\Mat_{n}^m)$ explicitly, so that we can make the computations we need.

\subsection{Action of $\gl(\Mat_{n}^m)$ on $\C[\Mat_{n}^m]$}
Let $V = \Mat_{n}^m$ be the space of $m$-tuples of $n \times n$ matrices. Let $E^{(i)}_{jk}$ denote the tuple of matrices with a $1$ in the $(j,k)^{th}$ entry of the $i^{th}$ matrix, and $0$'s everywhere else. Let $x^{(i)}_{jk}$ denote the coordinate function corresponding to the $(j,k)^{th}$ entry of the $i^{th}$ matrix. Then $\C[V] = \C[x^{(i)}_{jk} : 1\leq i \leq m, 1 \leq j,k \leq n]$.

Since $V$ is $mn^2$ dimensional, we can identify $\gl(V) = \gl(\Mat_{n}^m)$ with $\Mat_{mn^2}$, but we will do so in a very specific way. We will think of $\gl(\Mat_{n}^m) = \gl(\C^m \otimes \Mat_{n}) = \Mat_{m} \otimes \Mat_{n^2}$. We have already been explicit in the earlier sections about how we view $\Mat_{n^2}$ as $\gl(\Mat_{n})$. 

We index the rows and columns of $mn^2 \times mn^2$ matrices by $\{(i,j,k): 1 \leq i \leq m, 1 \leq j,k \leq n\}$ in lexicographic order. Thus, we can write $M \in \gl(\Mat_{n}^m) = \Mat_{mn^2}$ as

$$
M = \begin{pmatrix} M_{11} & \dots & M_{1m} \\
\vdots & \ddots & \vdots \\
M_{m1} & \dots & M_{mm}
\end{pmatrix},
$$
where each $M_{pq}$ is an $n^2 \times n^2$ matrix. Equivalently, we can write $M = \sum_{1 \leq p,q \leq m} E_{pq} \otimes M_{pq}$. With this indexing, we have the intuitive formula
$$
E_{pq} \otimes E_{ab,cd} = E_{pab, qcd}.
$$

The twisted action of the matrix $E_{pab,qcd} \in \gl(\Mat_{n}^m)$ is via the derivation $x^{(p)}_{ab} \partial^{(q)}_{cd}$ where $\partial^{(q)}_{cd}$ denotes the partial derivative with respect to the coordinate $x^{(q)}_{cd}$.

Recall the $\N^m$-grading on $\C[\Mat_{n}^m]$. Observe that the twisted action of $E_{pab,qcd}$ maps  $\C[V]_e$ to $\C[V]_{e + \delta_p - \delta_q}$, where $\delta_i = (0,\dots,0,\underbrace{1}_i,0\dots,0) \in \N^m$. Thus, we give an $\N^m$-grading on $\gl(\Mat_{n}) = \Mat_{mn^2}$.

\begin{definition} [Grading on $\gl(\Mat_{n}^m)$]
We give a grading on the Lie algebra $\gl(\Mat_{n}^m)$ by setting $\deg(E_{pab,qcd}) = \delta_p - \delta_q$.
We have
$$
\gl(\Mat_{n}^m) = \gl(\Mat_{n}^m)_0 \bigoplus_{p \neq q} \gl(\Mat_{n}^m)_{\delta_p - \delta_q}.
$$
For any $M  \in \Mat_{mn^2} = \gl(\Mat_{n}^m)$, we write $M = \sum_{1 \leq p,q \leq m} E_{pq} \otimes M_{pq}$, with $M_{pq} \in \Mat_{n^2}$. Then, the degree $0$ part is $\sum_{p = 1}^m E_{pp} \otimes M_{pp}$, and for $p \neq q$, $E_{pq} \otimes M_{pq}$ is the degree $\delta_p - \delta_q$ part. Thus the decomposition of $M$ into homogenous components is 
$$
M = (\sum_{p=1}^m E_{pp} \otimes M_{pp} ) \bigoplus_{p \neq q} E_{pq} \otimes M_{pq}.
$$
\end{definition}

The following lemma is immediate from the preceding discussion:

\begin{lemma} \label{L-grade-nm}
Let $M \in \Mat_{mn^2}$ be homogenous of degree $e$, and $f \in \C[\Mat_{n}^m]$ be homogenous of degree $e'$. Then $M \cd f$ is homogenous of degree $e + e'$.
\end{lemma}

\subsection{Computing the Lie algebra of symmetries}
For this subsection, let $M \in \gl(\Mat_{n}^m) = \Mat_{mn^2}$ be such that $M \cd (I_S)_n \subseteq (I_S)_n$. Further, write $M = \sum_{p,q} E_{pq} \otimes M_{pq}$, where $M_{pq} \in \Mat_{n^2}$, i.e.,
$$
M = \begin{pmatrix} M_{11} & \dots & M_{1m} \\
\vdots & \ddots & \vdots\\
M_{m1} & \dots & M_{mm}
\end{pmatrix}.
$$

\begin{lemma}
For $p \neq q$, we have $M_{pq} = \lambda_{pq} {\rm I}_{n^2}$ for some scalar $\lambda_{pq} \in \C$.
\end{lemma}

\begin{proof}
Without loss of generality, let us assume that $p = 2, q = 1$. We know that $M \cd f_{(n,0,\dots,0)} \in (I_S)_n$. Consider the degree $(n-1,1,0,\dots,0)$ homogenous part of $M \cd f_{(n,0,\dots,0)}$. By Lemma~\ref{L-grade-nm} and the description of the grading on $\gl(\Mat_{n}^m)$, we know that the degree $(n-1,1,0,\dots,0)$ homogenous part of $M \cd f_{(n,0,\dots,0)}$ is $(E_{21} \otimes M_{21}) \cd f_{(n,0,\dots,0)}$. Since, this must be in $(I_S)_{(n-1,1,0,\dots,0)}$ which is spanned by $f_{(n-1,1,0,\dots,0)}$, we must have
$$
(E_{21} \otimes M_{21}) \cd f_{(n,0,\dots,0)} = c \cdot f_{(n-1,1,0,\dots,0)},
$$
for some $c \in \C$. $M_{21} \in \Mat_{n^2}$ (and recall that we index the rows and columns of $\Mat_{n^2}$ by $[n] \times [n]$ in lexicographic order). Let the $(ab,cd)^{th}$ entry of $M_{21}$ be $\mu_{ab,cd}$.

Thus, 
\begin{align} \label{M21}
(E_{21} \otimes M_{21}) \cd f_{(n,0,\dots,0)}  &= \left(\sum_{a,b,c,d} \mu_{ab,cd} \cdot x^{(2)}_{ab} \partial^{(1)}_{cd} \right) \left(\sum_{\sigma \in S_n} \sgn(\sigma) \cdot x^{(1)}_{1 \sigma(1)} \dots x^{(1)}_{n\sigma(n)} \right) \\
\label{M22}
& = \sum_{\begin{array}{c} a,b,c,d,\sigma \\ \sigma(c) = d \end{array}} \mu_{ab,cd} \cdot \sgn(\sigma) \cdot x^{(1)}_{1 \sigma(1)} \dots x^{(1)}_{(c-1),\sigma(c-1)} x^{(2)}_{ab} x^{(1)}_{(c+1)\sigma(c+1)} \dots x^{(1)}_{n\sigma(n)}.
\end{align}

Since $(E_{21} \otimes M_{21}) \cd f_{(n,0,\dots,0)} = c \cdot f_{(n-1,1,0,\dots,0)}$, we will match the coefficients of monomials on both sides to get conditions on the entries of $M_{21}$.

First, fix $\pi \in S_n$, $i \in [n]$ and let $\alpha \beta \neq i \pi(i)$. Now, consider the coefficient of the monomial 
$$
m = x^{(1)}_{1 \pi(1)} \dots x^{(1)}_{(i-1),\sigma(i-1)} x^{(2)}_{\alpha,\beta} x^{(1)}_{(i+1)\sigma(i+1)} \dots x^{(1)}_{n\sigma(n)}.
$$

In the expression Equation~\ref{M22}, let us see what choices of $a,b,c,d,\sigma$ lead to this monomial. Clearly, we need $c = i$, $ d = \sigma(i)$. Moreover, we need $\sigma(k) = \pi(k)$ for all $k \neq i$, so $\sigma = \pi$ (and hence $\sigma(i) = \pi(i) = d$). Finally also observe that we also need $a = \alpha$ and $b = \beta$. Thus, there is precisely one choice for which can lead to the monomial $m$, and this means that the coefficient of the monomial $m$ is $ \mu_{\alpha\beta,i\pi(i)}\cdot \sgn(\pi)$. The coefficient of $m$ in $c \cdot f_{(n-1,1,0,\dots,0)}$ is zero, so we must have $\mu_{\alpha\beta,i\pi(i)} = 0$. Observe that as long as $i_1j_1 \neq i_2j_2$, we can choose $\alpha = i_1, \beta = j_1$, $i  = i_2$ and $\pi$ such that $\pi(i) = j_2$ and satisfy the condition $\alpha\beta \neq i\pi(i)$. Thus, all the off-diagonal terms of $M_{21}$ are zero. In other words $M_{21}$ is a diagonal matrix.

Now that $M_{21}$ is a diagonal matrix, we have
\begin{align*}
(E_{21} \otimes M_{21}) \cd f_{(n,0,\dots,0)} & = 
\sum_{\begin{array}{c} a,b,\sigma \\ \sigma(a) = b \end{array}} \mu_{ab,ab} \cdot \sgn(\sigma)  \cdot x^{(1)}_{1 \sigma(1)} \dots x^{(1)}_{(a-1),\sigma(a-1)} x^{(2)}_{ab} x^{(1)}_{(a+1)\sigma(a+1)} \dots x^{(1)}_{n\sigma(n)} \\
\label{M21-diag}
& = \sum_{a,\sigma} \mu_{a\sigma(a),a\sigma(a)} \cdot \sgn(\sigma) \cdot x^{(1)}_{1 \sigma(1)} \dots x^{(1)}_{(a-1),\sigma(a-1)} x^{(2)}_{a\sigma(a)} x^{(1)}_{(a+1)\sigma(a+1)} \dots x^{(1)}_{n\sigma(n)}.
\end{align*}

On the other hand 
\begin{align*}
c \cdot f_{(n-1,1,0,\dots,0)} = \sum_{a,\sigma} c \cdot  \sgn(\sigma) \cdot x^{(1)}_{1 \sigma(1)} \dots x^{(1)}_{(a-1),\sigma(a-1)} x^{(2)}_{a\sigma(a)} x^{(1)}_{(a+1)\sigma(a+1)} \dots x^{(1)}_{n\sigma(n)}.
\end{align*}

Thus, by matching coefficients of monomials, we get that $\mu_{a\sigma(a),a\sigma(a)} = c$. Since this is true for all choices of $a$ and $\sigma$, we have that $\mu_{ab,ab} = c$ for all $ab \in [n] \times [n]$. This means that $M_{21} = c {\rm I}_{n^2}$. Thus, with $\lambda_{21} = c$, we have $M_{21} = \lambda_{21} \cdot {\rm I}_{n^2}$ as required.
\end{proof}

Recall the action of $(\GL_n \times \GL_n)^{\times m}$ on $\Mat_{n}^m$ in Section~\ref{sec:intermediate}. This gives a homomorphism $\rho: (\GL_n \times \GL_n)^{\times m} \rightarrow \GL(\Mat_{n}^m) = \GL_{mn^2}$. In coordinates the map is given explicitly by the formula
$$
((P_1,Q_1),(P_2,Q_2),\dots,(P_m,Q_m)) \mapsto \sum_i E_{ii} \otimes P_i \otimes Q_i.
$$

Differentiating gives a Lie algebra homomorphism $d\rho: (\gl_n \times \gl_n)^{\times m} \rightarrow \gl(\Mat_{n}^m) = \Mat_{mn^2}$. Explicitly in coordinates, this is given by the formula
$$
((A_1,B_1),\dots,(A_m,B_m)) \mapsto \sum_i E_{ii} \otimes (A_i \otimes {\rm I}_n + {\rm I}_n \otimes B_i).
$$

Recall the group $H$ defined in Section~\ref{sec:intermediate}. By Lemma~\ref{L-inter}, we have $H = \{g \in (\GL_n \times \GL_n)^{\times m}\ |\ g (I_S)_n \subseteq (I_S)_n\} = \{g \in (\GL_n \times \GL_n)^{\times m}\ | \rho(g) \in \mathcal{G}_S\}$. It follows that 
$$
\Lie(H) = \{N \in (\gl_n \times \gl_n)^{\times m} \ |\ d\rho(N) \cdot (I_S)_n \subseteq (I_S)_n\} = \{N \in (\gl_n \times \gl_n)^{\times m} \ |\ d\rho(N) \in \g_S\}.
$$
The first equality essentially follows from the same argument in Proposition~\ref{P-Liealg-all}, and the second equality is clear from Proposition~\ref{P-Liealg-all}. From the description of $H$ in Section~\ref{sec:intermediate}, a straightforward computation gives
$$
d\rho (\Lie(H)) = \{{\rm I}_m \otimes A \otimes {\rm I}_n + {\rm I}_m \otimes {\rm I}_n \otimes B + D \otimes {\rm I}_{n} \otimes {\rm I}_n \ |\ A,B \in \Mat_{n}, D \in \Mat_{n} \text{ diagonal matrix}\}.
$$

\begin{lemma}
Consider the degree $0$ part of $M$, i.e., $M_0 = \sum_i E_{ii} \otimes M_{ii}$. Then $M_0 \in d\rho (\Lie(H))$
\end{lemma}

\begin{proof}
Let $M = \bigoplus_e M_e$ be its graded decomposition. We know that $M_{\delta_i - \delta_j} = E_{ij} \otimes M_{ij}$ for $i \neq j$, and $M_0 = \sum_i E_{ii} \otimes M_{ii}$. For all other $e$, $M_e = 0$. In particular, $M \cd (I_S)_n \subseteq (I_S)_n$ implies that for all $e$ with $\sum_i e_i = n$, $M_0 \cd f_e = \gamma_e f_e$ for some $\gamma_e \in \C$.

First, observe that $M_0 \cd f_{(n,0,\dots,0)} = M \cd \det(X_1) = \sum_i (E_{ii} \otimes M_{ii}) \cd (\det(X_1)) = (E_{11} \otimes M_{11}) \cd \det(X_1)$. This means that $(E_{11} \otimes M_{11}) \cd \det(X_1) = \gamma_{(n,0,\dots,0)} \det(X_1)$. By Theorem~\ref{theo:frob}, and Corollary~\ref{C-loscom}, we get that $M_{11}$ is of the form $A_1 \otimes {\rm I}_n + {\rm I}_n \otimes B$ for some $A,B \in \Mat_{n}$. Similarly, each $M_{ii}$ is of the form $A_i \otimes {\rm I}_n + {\rm I}_n \otimes B_i$. Thus $M_0$ is in the image of $d\rho$, since $M_0 = \sum_i E_{ii} \otimes (A_i \otimes {\rm I}_n + {\rm I}_n \otimes B_i)$ (see the explicit formula for $d\rho$ above).
 
 %Also note that $\Theta(M_0) = -M_0^t = \sum_i E_{ii} \otimes (-M_{ii}^t)$ is also in the image of $d\rho$ because the image of $d\rho$ is closed under $\Theta$ (straightforward to check).

This means that $M_0$ is in the image of $d\rho$ such that $(M_0) \cd (I_S)_n \subseteq (I_S)_n$. From the description of $\Lie(H)$ above, we get that $\Theta(M_0) \in d\rho (\Lie(H))$. But since $d\rho(\Lie(H))$ is closed under $\Theta$, we get that $M_0 \in d \rho (\Lie(H))$. 
\end{proof}

Thus, putting the above two lemmas together, we get that $M$ is of the form $C \otimes {\rm I}_n \otimes {\rm I}_n + {\rm I}_m \otimes A \otimes {\rm I}_n + {\rm I}_m \otimes {\rm I}_n \otimes B$, i.e., $M \in \Lie(G_{n,m})$.

Thus, we conclude that
$$
\g_S = \Theta \{M \ |\ M \cd (I_S)_n \subseteq (I_S)_n\} \subseteq  \Theta(\Lie(G_{n,m})) = \Lie(G_{n,m}).
$$
since $\Lie(G_{n,m})$ is closed under $\Theta$. 

The reverse inclusion is clear since $G_{n,m} \subseteq \mathcal{G}_S$ implies that $\Lie(G_{n,m}) \subseteq \Lie(\mathcal{G}_S) = \g_S$. So, we conclude that
$$
\g_S = \Lie(G_{n,m}).
$$
Further, this implies (by the Lie subgroups -- Lie subalgebras correspondence) that 
$$
\mathcal{G}_S^\circ = G_{n,m}.
$$

Let us record this result. 

\begin{corollary}
Let $S = \SING_{n,m} \subseteq V = \Mat_{n}$. Then the connected group of symmetries
$$
\mathcal{G}_S^\circ = G_{n,m}.
$$
\end{corollary}

In the next subsection, we will determine the entire group of symmetries. The argument is very similar to the one in Section~\ref{sec:frob}

\subsection{The group of symmetries}
From the above discussion, we know that $G_{n,m}$ is the identity component of $\mathcal{G}_S$. The component group $\mathcal{G}_S / G_{n,m}$ is a finite group\footnote{Component groups are always finite for linear algebraic groups.}. In any case the fact that $G_{n,m}$ is the identity component of $\mathcal{G}_S$ means that $\mathcal{G}_S$ normalizes $G_{n,m}$. Thus $\mathcal{G}_S \subseteq N_{\GL(\Mat_{n}^m)} (G_{n,m})$. Let us therefore compute this normalizer. 

Consider the transpose map $\tau: \Mat_{n}^m \rightarrow \Mat_{n}^m$ given by $(X_1,\dots,X_m) \mapsto (X_1^t,\dots,X_m^t)$. Viewing $\Mat_{n}^m$ as $\C^m \otimes \C^n \otimes \C^n$, the $\tau$ is simply the linear map that switches the second and third tensor factors. When $m = n$, then all three tensor factors are $\C^n$, and there are linear transformations that permute them in any way. For a permutation $\sigma \in S_3$, let us denote by $\tau_{\sigma}$ the corresponding linear map. Precisely, we have
\begin{align*}
\tau_\sigma: \C^n \otimes \C^n \otimes \C^n &\longrightarrow \C^n \otimes \C^n \otimes \C^n \\
\sum_i v_{i,1} \otimes v_{i,2} \otimes v_{i,3} &\longmapsto \sum_i v_{i,\sigma(1)} \otimes v_{i,\sigma(2)} \otimes v_{i,\sigma(3)}.
\end{align*}

In particular, the transpose morphism $\tau = \tau_{\sigma}$ for $\sigma$ defined as $\sigma(1) = 1, \sigma(2) = 3, \sigma(3) = 2$.

Let us define 
$$
\Sigma_{n,m} = \begin{cases} \{1,\tau\} & \text{ if $n \neq m$} \\
\{\tau_{\sigma}: \sigma \in S_3\} & \text{ if $n = m$}. \end{cases}
$$

Observe that $\Sigma_{n,m}$ is a subset of linear transformations of $\C^m \otimes \C^n \otimes \C^n$. When $m = n$, $\Sigma_{n,m}$ consists of six linear transformations, and when $m \neq n$, it consists of two linear transformations.

\begin{lemma}
The normalizer
$$
N_{\GL(\Mat_{n}^m)}(G_{n,m}) = \{h_1h_2  \ | h_1 \in G_{n,m}, h_2 \in \Sigma_{n,m} \} = G_{n,m} \rtimes \Sigma_{n,m}.
$$
\end{lemma}

\begin{proof}
The argument is very similar to the one for the $m=1$ case. Let $g \in \GL(\Mat_{n}^m)$ be such that $g$ normalizes $G_{n,m}$. Thus, it normalizes its derived group which is isomorphic to $\SL_m \times \SL_n \times \SL_n$, and hence the Lie algebra of its derived group. This Lie algebra is $\sll_m \times \sll_n \times \sll_n$ which embeds in $\gl(\Mat_{n}^m)$ as $$
\{C \otimes {\rm I}_n \otimes {\rm I}_n + {\rm I}_m \otimes A \otimes {\rm I}_n + {\rm I}_m \otimes {\rm I}_n \otimes B\ | \  C \in \sll_m, A,B, \in \sll_n\}.
$$
For simplicity we will continue to refer to this Lie subalgebra as $\sll_m \times \sll_n \times \sll_n$. $g$ normalizes this Lie subalgebra. This Lie subalgebra has exactly $3$ simple ideals, so the (conjugation) action of $g$ has to permute them. There are two cases, when $m = n$, then there are 6 possible permutations, and when $m \neq n$, the $\sll_m$ must remain fixed and the two $\sll_n$'s can be permuted. Now, one observes (in both cases) that for some $h_2 \in \Sigma_{n,m}$, the action of $g' = gh_2$ (by conjugation) fixes the three simple ideals. 

Thus, we have $g' \in \GL(\Mat_{n}^m) = \subseteq \Mat_{mn^2} = \Mat_{m} \otimes \Mat_{n^2}$. Observe that in this decomposition, we identify $\Mat_{m}$ with linear transformations on $\C^m$ (the first tensor factor) and $\Mat_{n^2}$ with linear transformations on $\C^n \otimes \C^n$ (the second and third tensor factors). Note that $g'$ fixes $\sll_m$, and write $g' = \sum_{i=1}^r P_i \otimes Q_i$ where $P_i \in \Mat_{m} = \gl(\C^m)$ and $Q_i \in \gl(\C^n \otimes \C^n) = \Mat_{n^2}$ such that $\{P_i\}$ is a linearly independent subset of $\Mat_{m}$ and $\{Q_i\}$ is a linearly independent subset of $\Mat_{n^2}$. 

The same argument as in  the proof of Lemma~\ref{L-normfrob} proves that $r = 1$ and $P_1 \in \GL_m$. Repeating the argument for the other tensor factors, we get that $g' = P_1 \otimes P_2 \otimes P_3$, where $P_1 \in \GL_m$, $P_2 \in \GL_n$ and $P_3 \in \GL_n$. In other words, $g' \in G_{n,m}$. 

This proves that $g = g'h_2^{-1} \in G_{n,m} \rtimes \Sigma_{n,m}$. This proves that $N_{\GL(\Mat_{n}^m)}(G_{n,m}) \subseteq G_{n,m} \rtimes \Sigma_{n,m}$. The reverse inclusion is clear.
\end{proof}

\begin{proof} [Proof of Theorem~\ref{theo:gos}]
It is clear that $G_{n,m} \subseteq \mathcal{G}_S \subseteq G_{n,m} \rtimes \Sigma_{n,m}$. Any algebraic group sandwiched between $G_{n,m}$ and $G_{n,m} \times \Sigma_{n,m}$ must be a union of components, i.e., $\mathcal{G}_S = \cup_{h \in I} G_{n,m} \cdot h$ for some subgroup $I \subseteq \Sigma_{n,m}$. But this subgroup is easy to determine. Clearly the transpose morphism $\tau$ is in $I$, so $I = \Sigma_{n,m}$ when $m \neq n$. 

Now, consider the case $m = n$. We still claim that $I = \{e,\tau\}$. Observe that $\{e,\tau\}$ is a proper maximal subgroup of $\Sigma_3$, and we have seen that $\{e,\tau\} \subseteq I$. Thus, it suffices to show that $I \subsetneq \Sigma_3$. 

To see this, let $\sigma \in S_3$ be the permutation $\sigma(1) = 2$, $\sigma(2) = 1$ and $\sigma(3) = 3$. Let us take $X = ({\rm I}_n,0,\dots,0)$. Then observe that $\tau_\sigma (X) = (E_{11},E_{12},\dots,E_{1n})$. Observe that $X \notin \SING_{n,m}$ whereas $\tau_\sigma(X) \in \SING_{n,m}$. Thus $\tau_\sigma \notin I$. This forces $\{e,\tau\} \subseteq I \subsetneq \Sigma_{n,m}$. Thus $I = \{e,\tau\}$. 

Thus irrespective of whether $m$ and $n$ are equal or not, we have $\mathcal{G}_S = G_{n,m} \rtimes \Z/2$ as required.
\end{proof}

\subsection{Symmetries of $\NSING_{n,m}$}
All the work for computing the symmetries of $\NSING_{n,m}$ has already been done, and we just need to put it together.

\begin{proof} [Proof of Theorem~\ref{theo:ngos}]
Let us denote by $I \subseteq \C[\Mat_{n}^m]$ the vanishing ideal of $\SING_{n,m}$ and by $J\subseteq \C[\Mat_{n}^m]$ the vanishing ideal of $\NSING_{n,m}$. Our first claim is that $I_n = J_n$ (see Lemma~\ref{L-N-ideal}). %Clearly $I_n \subseteq J_n$ because $\NSING_{n,m} \subseteq \SING_{n,m}$. Note that both $\NSING_{n,m}$ and $\SING_{n,m}$ are stable under the action of $\GL_m \times \GL_n \times \GL_n$, and hence so are $I_n$ and $J_n$. So, $J_n$ has to be a direct sum of irreducibles. On the other hand, we show in Appendix~\ref{App.rep} that for any irreducible representation outside $I_n$, a highest weight vector of that irreducible will not vanish on a particular linear subspace of $\SING_{n,m}$. The crucial point to observe is that the linear subspace that we pick is actually contained in $\NSING_{n,m}$\footnote{In fact, this linear subspace is an example of what we call a ``coordinate subspace", and any coordinate subspace contained in $\SING_{n,m}$ must be contained in $\NSING_{n,m}$ which we will show in the next section.}. To summarize, we have $I_n = J_n$. 

Thus the lie algebra of symmetries for $\NSING_{n,m}$ is a subalgebra of $\{M \in \gl(\Mat_{n}^m)\ |\ M I_n \subseteq I_n\} = \g_{n,m}$. Thus, the connected group of symmetries for $\NSING_{n,m}$ is a subgroup of $G_{n,m}$. But clearly $G_{n,m}$ preserves $\NSING_{n,m}$. Thus, the connected group of symmetries for $\NSING_{n,m}$ is also $G_{n,m}$. To determine the component group, the same analysis as in the previous subsection works. Thus the group of symmetries for $\NSING_{n,m}$ is exactly the same as the group of symmetries for $\SING_{n,m}$.
\end{proof}

\section{Singular tuples of matrices cannot be a null cone} \label{sec:notnullcone}
In this section, we will prove our main theorem, i.e., Theorem~\ref{theo:nullcone} as well as Theorem~\ref{theo:nullcone2}. To do so, we need to understand the coordinate subspaces (see Definition~\ref{D-coordsubspace}) of $\NSING_{n,m}$ and $\SING_{n,m}$. The main point is that both $\NSING_{n,m}$ and $\SING_{n,m}$ have exactly the same coordinate subspaces. First a few definitions. 

\begin{definition} [Support of a matrix]
For a matrix $M$, it support $\Supp(M) \subseteq [n] \times [n]$ is defined as the subset of positions with non-zero entries. In other words, $(j,k) \in \Supp(M)$ if and only if the $(j,k)^{th}$ entry of $M$ is non-zero.
\end{definition}

\begin{definition} [Support and union support of a tuple of matrices]
For $X = (X_1,\dots,X_n) \in \Mat_{n}^m$, we define its support $\Supp(X) \subseteq [m] \times [n] \times [n]$ as the subset of positions with non-zero entries. More precisely $\Supp(X)$ consists of all $(i,j,k)$ such that the $(j,k)^{th}$ coordinate of $X_i$ is non-zero.

We also define its union support $\USupp(X) \subseteq [n] \times [n]$ to be $\cup_i \Supp(X_i)$. In other words, $(j,k) \in \USupp(X)$ if and only if the $(j,k)^{th}$ entry of some $X_i$ is non-zero. 
\end{definition}

Let us define a map 
\begin{align*}
\pi_{2,3} : [m] \times [n] \times [n] & \longrightarrow [n] \times [n] \\
(i,j,k) & \longmapsto (j,k)
\end{align*}

\begin{remark}
For $X = (X_1,\dots,X_m) \in \Mat_{n}^m$, the union support $\USupp(X)$ can also be seen in the following equivalent ways
\begin{enumerate}
\item $\pi_{2,3} (\Supp(X))$;
\item  $\Supp(\sum_i t_i X_i)$ for indeterminates $t_1,\dots, t_m$; 
\item $\Supp(\sum_i c_i X_i)$ for generic $c_i \in \C$.
\end{enumerate}
\end{remark}

Recall that on $V = \Mat_{n}^m$, we denote by $x^{(i)}_{j,k}$ the $(j,k)^{th}$ coordinate of the $i^{th}$ matrix. 

\begin{definition}
For $I \subseteq [m] \times [n] \times [n]$, we define the linear subspace of $\Mat_{n}^m$ 
$$
L_I = \{X \in \Mat_{n}^m \ |\ \Supp(X) \subseteq I\}.
$$
Equivalently, it can be seen as the zero locus of $\{x^{(i)}_{j,k}\ |\ (i,j,k) \notin I\}$.
\end{definition}

\subsection{Coordinate subspaces of $\NSING_{n,m}$ and $\SING_{n,m}$}
The following result will be derived easily from well known characterizations of $\NSING_{n,m}$. We say a subsset $J \subseteq [n] \times [n]$ contains a permutation $\sigma \in S_n$ if $\{(i,\sigma(i)) \ |\ 1 \leq i \leq n\} \subseteq J$. We say $J \subseteq [n] \times [n]$ is {\em permutation free} if it does not contain any permutation.

\begin{proposition}
For $I \subseteq [m] \times [n] \times [n]$, $L_I \subseteq \NSING_{n,m}$ if and only if $\pi_{2,3}(I) \subseteq [n] \times [n]$ is permutation free.
\end{proposition}

\begin{proof}
Let us recall that $\NSING_{n,m}$ is the null cone for the action of $\SL_n \times \SL_n$. Let $T = \ST_n \times \ST_n$ denote the (standard) maximal torus, i.e., $T$ consists of pairs of diagonal matrices with determinant $1$. Recall that $\NSING_{n,m} = (\SL_n \times \SL_n) \cdot \mathcal{N}_T(\Mat_{n}^m)$ by Theorem~\ref{theo:HM}. Further, from the description of the null cone for tori in Section~\ref{sec:inv.thry.tori}, it can be deduced that
$$
\mathcal{N}_T(\Mat_{n}^m) = \bigcup_{\pi_{2,3}(I) \text{ permutation free}} L_I.
$$
Another simple way to see this is to understand that the invariant ring is generated by monomials of the form $\prod_{(i,j,k) \in J} x^{(i)}_{j,k}$ where $|J| = n$ and $\pi_{2,3}(J)$ is a permutation. Thus, we conclude that $L_I \subseteq \NSING_{n,m}$ if $\pi_{2,3}(I)$ is permutation free. Alternately, one can see from the description of $\SL_n \times \SL_n$ invariants (say for example \cite[Theorem~1.4]{DM}) that all non-constant homogenous invariants vanish on $L_I$.

Conversely, suppose $\pi_{2,3}(I)$ is not permutation free. So, $\pi_{2,3}(I)$ must contain some permutation, say $\sigma$. Thus for all $1 \leq i \leq n$, there exists  $p_i \in [m]$, such that $(p_i,i,\sigma(i)) \in I$. Let $X = (X_1,\dots,X_m) \in \Mat_{n}^m$ be such that that $(i,\sigma(i))^{th}$ entry of $X_{p_i}$ is $1$ and all other entries are zero. Clearly $X \in L_I$ and further $\sum_i X_i$ is a permutation matrix (the one associated to $\sigma$), and hence non-singular. But this means that $X \notin \SING_{n,m}$. Hence $L_I \nsubseteq \SING_{n,m}$, and so $L_I \nsubseteq \NSING_{n,m}$ (because $\NSING_{n,m} \subseteq \SING_{n,m}$).
\end{proof}

Indeed, observe that the proof of above also gives the following:

\begin{proposition}
For $I \subseteq [m] \times [n] \times [n]$, $L_I \subseteq \SING_{n,m}$ if and only if $\pi_{2,3}(I) \subseteq [n] \times [n]$ is permutation free.
\end{proposition}

Thus we get the following corollary that is crucial for our purposes.

\begin{corollary} \label{C-coordsubs}
For $I \subseteq [m] \times [n] \times [n]$, $L_I \subseteq \NSING_{n,m}$ if and only if $L_I \subseteq \SING_{n,m}$.
\end{corollary}

\subsection{Proof of main result}
First, let us prove Theorem~\ref{theo:nullcone2}.

\begin{proof}
Let $G$ be a reductive group acting on $V = \Mat_{n}^m$ preserving $S = \SING_{n,m}$ such that $\mathcal{N}_G(V) \subseteq S$. This action is given by a map $\rho: G \rightarrow \GL(V)$. The fact that $G$ preserves $\SING_{n,m}$ means that the image $\rho(G)$ is contained in the group of symmetries $\mathcal{G}_S = G_{n,m} \times \Z/2$. Now, consider a maximal torus $T$ of $G$. Then $\rho(T)$ is a subtorus of $\rho(G)$ and hence a subtorus of $G_{n,m}$. Thus, $\rho(T)$ is contained in a maximal torus of $G_{n,m}$ and all maximal tori are conjugate under the action of $G_{n,m}$. Thus, for some $g \in G_{n,m}$, we have that $g \rho(T) g^{-1}$ is a subtorus of the standard maximal torus $T_{n,m}$. The standard maximal torus 
$$
T_{n,m} = \{D_1 \otimes D_2 \otimes D_3\ |\ D_1 \in T_m, D_2 \in T_n, D_3 \in T_n\},
$$
where $T_k$ denotes the (standard) diagonal torus of $\GL_k$. Let $\widetilde{\rho}: G \rightarrow \GL(V)$ be defined by $\widetilde{\rho}(h) = g \rho(h) g^{-1}$. This is also an action that satisfies the hypothesis, in particular, $\mathcal{N}_{G,\widetilde{\rho}}(V) = g \cdot \mathcal{N}_{G,\rho}(V)$, and has the added feature that $\widetilde{\rho}(T) \subseteq T_{n,m}$. The point of the above discussion was to establish the fact that the standard basis $\{E_{ijk}\}$ of $V$ is a weight basis for the action defined by $\widetilde{\rho}$ (since it is a weight basis for $T_{n,m}$). Thus, the null cone for the torus $\mathcal{N}_{T,\widetilde{\rho}}(V)$ is a union of certain coordinate subspaces of $\SING_{n,m}$, and hence contained in $\NSING_{n,m}$ by Corollary~\ref{C-coordsubs}. Thus, the null cone $\mathcal{N}_{G,\widetilde{\rho}}(V) = \widetilde{\rho}(G) \cdot \mathcal{N}_{T,\widetilde{\rho}}(V) \subseteq \NSING_{n,m}$ because $\widetilde{\rho}(G) = g \rho(G) g^{-1} \subseteq G_{n,m} \rtimes \Z/2$, which is the group of symmetries of $\NSING_{n,m}$.

Now, we simply observe that $\mathcal{N}_{G,\rho}(V) = g^{-1} \cdot \mathcal{N}_{G,\widetilde{\rho}}(V)  \subseteq g^{-1} \NSING_{n,m} = \NSING_{n,m}$, which is the required conclusion.
\end{proof}

Before proving Theorem~\ref{theo:nullcone}, let us quickly recollect the fact that $\NSING_{n,m}$ is a proper subset of $\SING_{n,m}$ precisely when $n,m \geq 3$. To begin, we refer the reader to \cite{GGOW16,IQS} for many equivalent characterizations of the $\NSING_{n,m}$ (we will not recall them here). First, if $n= 1$ or $m = 1$, it is obvious. For $n = 2$, $\NSING_{2,m} = \SING_{2,m}$ follows from the fact that for $2 \times 2$ linear matrices, their commutative rank and non-commutative rank are the same (this follows from \cite[Remark~1]{FR} or \cite[Lemma~2.9]{DM}). For $m = 2$, it follows from the fact that polynomials of the form $\det(c_1 X_1 + c_2X_2)$ generates the invariant ring for the action of $\SL_n \times \SL_n$, see \cite{Happel1,Happel2}. Thus $\NSING_{n,2}$ is the zero locus of $\{\det(c_1 X_1 + c_2 X_2) : c_i \in \C\}$, which is precisely $\SING_{n,2}$.

On the other hand, for $n = m = 3$, the $3$-tuple 
$$
X = \left( \begin{pmatrix} 0 & 1 & 0 \\ -1 & 0 & 0 \\ 0 & 0 & 0 \end{pmatrix}, \begin{pmatrix} 0 & 0 & 1 \\ 0 & 0 & 0 \\ -1 & 0 & 0  \end{pmatrix}, \begin{pmatrix} 0 &  0 & 0 \\ 0 &0 & 1 \\ 0& -1 & 0 \end{pmatrix} \right) \in \Mat_{3}^3
$$
is in $\SING_{3,3}$ but not in $\NSING_{3,3}$ (see for example \cite{FR} or \cite[Example~1.1]{DM-ncrk}). For larger $n$ and $m$, this example can be modified in straightforward ways to show that $\NSING_{n,m}$ is a proper subset of $\SING_{n,m}$.

\begin{proof} [Proof of Theorem~\ref{theo:nullcone}]
Since $n,m \geq 3$, we know that $\NSING_{n,m} \subsetneq \SING_{n,m}$ by the above discussion. Suppose there was a group $G$ acting on $V = \Mat_{n}^m$ such that the null cone is $\SING_{n,m}$. This means in particular that $G$ must preserve $\SING_{n,m}$. Thus, we can apply Theorem~\ref{theo:nullcone2} to deduce that the null cone is contained in $\NSING_{n,m}$, which is a contradiction. 
\end{proof}

Finally, let us reiterate that if $n$ or $m$ is less than $3$, then $\SING_{n,m} = \NSING_{n,m}$ is the null cone for the left-right action of $\SL_n \times \SL_n$.

\section{The ring generated by determinantal polynomials is not invariant  for {\em any} group action} \label{sec:inv.conv}
Let $D(n,m)$ denote the ring $\C[\{\det(\sum_i c_i X_i) : c_i \in \C\}] \subseteq \C[\Mat_{n}^m]$. We want to show that this is not the invariant ring for the linear action of {\em any} group on $V = \Mat_{n}$. First, it suffices to restrict ourselves to subgroups of $V$. Indeed, if there was such a group $G$ with an action, i.e., a map $\rho:G \rightarrow \GL(V)$, then the ring of invariants for the action of $G$ is the same as the ring of invariants for the action of $\rho(G)$ which is a subgroup of $\GL(V)$.

Let us look at the subgroup $G_D \subseteq \GL(V)$ consisting of all linear transformations that leave $\det(\sum_i c_i X_i)$ invariant for all choices of $c_i \in \C$, i.e.,  
$$
G_D = \{g \in \GL(V) = \GL_{mn^2} \ |\ g \cdot \det(\sum_i c_i X_i) = \det(\sum_i c_i X_i) \ \forall\ c_i \in C\}.
$$

Let us also define
$$
G_{\rm det} = \{g \in \GL(\Mat_{n}) = \GL_{n^2} \ |\ g \cdot \det = \det\}.
$$

Recall from Frobenius that $G_{n,1} \rtimes \Z/2$ is the group of symmetries of $\SING_{n,1}$. If we define $\SL_{n,1} = \{A \otimes B\ |\ A,B \in \SL_n\} \subseteq G_{n,1}$, then it follows easily that
$$
G_{\rm det} = \SL_{n,1} \rtimes \Z/2.
$$

We will prove the following proposition:

\begin{proposition}
The group $G_D = \{ {\rm I}_m \otimes C \ | C \in G_{\rm det} \}$.
\end{proposition}

\begin{proof}
Let $g \in G \subseteq \GL_{mn^2}$. Then write 
$$
g = \begin{pmatrix} 
g_{11} & g_{12} & \dots & g_{1m} \\
\vdots & \ddots & \ddots & \vdots\\
g_{m1} & \dots & \dots &  g_{mm}\\
\end{pmatrix},
$$
where each block $g_{ij}$ is an $n^2 \times n^2$ matrix (which describes the map from the $j^{th}$ copy of $\Mat_{n}$ to the $i^{th}$ copy of $\Mat_{n}$). We will denote the action of $g$ by $\ast$ to avoid confusion with matrix multiplication.

Now, $g$ (and $g^{-1}$) fixes $\det(X_1)$, so we have $\det(X_1) = \det(g \ast X_1)$. Observe that $g \ast (X_1,0,\dots,0) = (g_{11} \ast X_1,\dots, g_{m1} \ast X_1)$. In particular, we must have $\det(X_1) = \det(g_{11} \ast X_1)$. Thus $g_{11} \in G_{\rm det}$. 

In particular, let 
$$
h = \begin{pmatrix} 
g_{11}^{-1} &  \dots & 0 \\
\vdots & \ddots &  \vdots\\
0 & \dots & g_{11}^{-1}\\
\end{pmatrix},
$$

Then since $g$ and $h$ preserve $\det(X_1)$, so does $gh$. Observe that we have 
$$
(gh \ast X)_1 = X_1 + g_{12}g_{11}^{-1} \ast X_2 + \dots + g_{1m}g_{11}^{-1} \ast X_n.
$$

Let us write 
$$
L = g_{12}g_{11}^{-1} \ast X_2 + \dots + g_{1m}g_{11}^{-1} \ast X_n.
$$

We now have $\det(X_1 + L) = \det(X_1)$, where $L$ is a matrix whose entries are linear functions in $(X_i)_{j,k}$ with $i \geq 2$. Suppose $L \neq 0$, then w.l.o.g, let us assume $L_{1,1} \neq 0$. When we expand out $\det(X_1 + L)$ with the definition as sum over all permutations, the term $L_{1,1}\cdot (X_1)_{2,2}  \cdot (X_1)_{3,3}\cdot \dots \cdot (X_1)_{n,n}$ occurs, and cannot be cancelled. This is because in no other permutation can we get the subterm $(X_1)_{2,2} \cdot (X_1)_{3,3} \cdot \dots \cdot (X_1)_{n,n}$. But then, this means that $\det(X_1) \neq \det(X_1 + L)$, which is a contradiction. Hence, $L = 0$. Since $g_{11} \in \GL_{n^2}$ is invertible, this means that $g_{12}, \dots, g_{1m} = 0$.

The argument above generalizes in the following way:
Consider the identification $\Mat_{n}^m = \Mat_{n} \otimes K^m$. Let the standard basis for this $K^m$ be $\{w_1,\dots,w_m\}$. Then the above argument simply says that $g$ preserves the ``slice'' $\Mat_{n} \otimes w_1$. Clearly, the same argument will show that $g$ preserves $\Mat_{n} \otimes w$ for all $w \in K^m$. Specializing to each $w_i$, we get that $g_{ij} = 0$ whenever $i \neq j$. So, we have 

$$
g =  \begin{pmatrix} 
g_{11} &  \dots & 0 \\
\vdots & \ddots &  \vdots\\
0 & \dots & g_{mm}\\
\end{pmatrix}.
$$

Now, suppose $g_{11} \neq g_{22}$. Again, w.l.o.g, we can assume column 1 of $g_{11} \neq$ column 1 of $g_{22}$. Then, $g \ast (E_{1,1} \otimes (w_1 + w_2)) \notin \Mat_{n} \otimes (w_1 + w_2)$. Note that when viewing an $n^2 \times n^2$ matrix (say $N$) as linear transformations on $\Mat_{n}$, the first column tells us the image of $E_{1,1}$ under $N$. Thus $g_{11} = g_{22}$. By a similar argument $g_{ii} = g_{11}$ for all $i$. To summarize, we have $g = {\rm I}_m \otimes g_{11}$ with $g_{11} \in G_{\rm det}$. Thus $G_D \subseteq \{ {\rm I}_m \otimes C \ | C \in G_{\rm det} \}$. The other inclusion is obvious. 
\end{proof}

\begin{proof} [Proof of Theorem~\ref{theo:invring}]
Clearly, $D(n,m) \subseteq \C[V]^{G_D}$. Recall the left-right action of $\SL_n \times \SL_n$ on $V$ given by $(A,B) \cdot (X_1,\dots,X_m) = (AX_1B^t,\dots, AX_mB^t)$. This action is given by a map $\rho: \SL_n \times \SL_n \rightarrow \GL(\Mat_{n})$. The image of a reductive group under a morphism of algebraic groups is reductive, so $\rho(\SL_n \times \SL_n)$ is a connected reductive subgroup of $\GL(V)$. We observe that in fact this is precisely the identity component of $G_D$. Thus $G_D$ is a reductive group. Hence, the null cone for $G_D$ is the same as the null cone for its identity component $\rho(\SL_n \times \SL_n)$, and hence equal to the null cone for $\SL_n \times \SL_n$, which we know is $\NSING_{n,m}$. Thus to summarize, the zero locus of all non-constant homogenous elements of $\C[V]^{G_D}$ is $\NSING_{n,m}$. On the other hand the zero locus of all non-constant homogenous elements of $D(n,m)$ is precisely $\SING_{n,m}$.  Thus, this means that we have a proper inclusion $D(n,m) \subsetneq \C[V]^{G_D}$.

Now, suppose there was {\em any} group $G$ such that $\C[V]^G = D(n,m)$. Suppose the action is given by $\rho:G \rightarrow \GL(V)$. Then $\rho(G) \subseteq G_D$ by the previous proposition, so $\C[V]^G = \C[V]^{\rho(G)} \supseteq \C[V]^{G_D}$. But this is a contradiction because $\C[V]^G = D(n,m) \subsetneq \C[V]^{G_D} \subseteq \C[V]^G$. Thus, there is no such group.
\end{proof}

\section{Discussion and open questions} \label{sec:disc}
This paper demonstrates another collaboration of different fields in mathematics. Expanding on ongoing work cited in the introduction, here too fundamental problems in computational complexity have given rise to a new flavor of problems that are purely algebraic in nature, some of which arise from analyzing analytic (rather than symbolic) algorithms. We feel that it is important to introduce these problems to representation theorists, algebraic geometers and commutative algebraists. The results of this paper open the door for several further avenues of research, inviting a further collaboration between theoretical computer scientists and mathematicians to resolve them. 

Let us begin with the stating that $\SING_{n,m}$ is a very important variety to study due to its connection to circuit lower bounds (\cite{KI}) that we mentioned earlier. Insights from any field of mathematics may be helpful! The major open problem is of course:

\begin{problem}
Is there a deterministic polynomial time algorithm for SDIT?
\end{problem}

Various subclasses of SDIT (and PIT) have polynomial time algorithms. For example, we say an $m$-tuple of $n \times n$ matrices $X = (X_1,\dots,X_m)$ satisfies the property $(R1)$ if the linear subspace in $\Mat_{n}^m$ spanned by $X_1,\dots,X_m$ has a basis consisting of rank $1$ matrices. It turns out that if $X$ satisfies $(R1)$, then $X \in \SING_{n,m}$ if and only if $X \in \NSING_{n,m}$. Thus, SDIT restricted to tuples with the $(R1)$ property can be solved via a null cone membership algorithm! (this is implicit in \cite{Gur04}). One direction of future research is to consider the following natural generalization of the $(R1)$ property. 

For fixed $k\in \Z_{\geq 1}$ We say $X = (X_1,\dots,X_m)$ satisfies the property $(Rk)$ if the linear subspace in $\Mat_{n}^m$ spanned by $X_1,\dots,X_m$ has a basis consisting of rank $\leq k$ matrices.

\begin{problem}
Is there a deterministic polynomial time algorithm for SDIT for tuples satisfying $(Rk)$? How about $(R2)$?
\end{problem}

Next, we turn to the symmetry group of an algebraic subvariety.

\begin{problem}
What algorithms can one use to determine the group of symmetries of a subvariety? How efficient are these algorithms?
\end{problem} 

In this paper, we explicitly determined the group of symmetries of one family of variety. It is however very clear that most steps are algorithmic. Roughly speaking, if the generators for the ideal of polynomials vanishing on the subvariety are given as an input, then determining the Lie algebra of symmetries reduces to solving a system of linear equations. So, in terms of the input size of such generating polynomials given by their coefficients, this Lie algebra part is efficient. It is not clear to us how to obtain the group itself efficiently from the Lie algebra. Moreover, if we are given the generating polynomials that describes the subvariety (set-theoretically) in an implicit, concise way (as in SING) it seems that more work is needed even to define the computational task. It is possible that when the generators themselves have some symmetries, or rich relations (as in SING), one can do more.

Another general problem to be pursued is to get a better understanding of null cones (and orbit closure equivalence classes)
\begin{problem}
Can one classify null cones? What features must a subvariety satisfy in order to possibly be a null cone?
\end{problem}

In this paper, we used mainly the fact that the null cone must be the translation (by a group element) of a union of coordinate subspaces (i.e., the Hilbert--Mumford criterion). It will be interesting to find other properties of null cones which distinguish them from arbitrary subvarieties.

A different direction to pursue is the following. The main result of this paper is that $\SING_{n,m}$ is not a null cone for any {\em reductive} group action. Natural as this condition is mathematically (and we use it and consequences of it here), it is not important algorithmically, and one can potentially implement and analyze null cone membership algorithms using non-reductive groups.\footnote{Clearly, for such groups the definition of a null cone must be take to be the analytic one.} So, what if we drop the reductivity assumption?
\begin{problem}
Can $\SING_{n,m}$ be the null cone for the action of a non-reductive group?
\end{problem}

Now, we mention a few more problems which are a little bit more technical, and of interest to commutative algebraists and algebraic geometers.

\begin{problem}
Let $I$ be the ideal of polynomials vanishing on $\SING_{n,m}$. Determine the ideal generators of $I$. Do the determinantal polynomials $\det(\sum_i c_i X_i)$ generate the ideal?
\end{problem}

\begin{problem}
Consider the ring $\C[\{\det(\sum_i c_i X_i) : c_i \in \C\}] \subseteq \C[\Mat_{n,n}^m]$. Is it Cohen--Macaulay? What is its regularity, etc?
\end{problem}

\bibliographystyle{alpha}
\bibliography{refs}

\appendix
\section{Missing proofs for Section~\ref{sec:gos}} \label{App.gos}
In this appendix, we will give the complete details of the theoretical ideas that go into Lemma~\ref{Lgos-graded} and Proposition~\ref{P-Liealg-all}. First, we note a lemma that will find repeated use.

\begin{lemma} \label{repeat}
Let $W$ be a (finite-dimensional) linear subspace and $U$ be a linear subspace of $W$. Then 
$$
\{g \in \GL(W)\ |\ gU \subseteq U\} = \{g \in \GL(W)\ |\ gU = U\}.
$$
\end{lemma}
\begin{proof}
The proof is straightforward and left to the reader.
\end{proof}

Let $S \subseteq V$ be a subvariety, and let $I_S$ denote the ideal of all polynomials that vanish on $S$. Recall that the action of $\GL(V)$ on $V$ gives an induced action on the polynomial ring $\C[V]$. Further, this action preserves the degree of the polynomials, so $\C[V]_{\leq d}$ (subspace of polynomials of degree $\leq d$) and $\C[V]_a$ (subspace of homogeneous polynomials of degree $a$) are subrepresentations (for any non-negative integers $d$ and $a$). Recall that $\mathcal{G}_S = \{g \in \GL(V)\ |\ g S = S\}$ is the group of symmetries.

\begin{lemma} \label{L-ode}
The group of symmetries
$$
\mathcal{G}_S = \{g \in \GL(V)\ | g \cdot I_S \subseteq I_S\} = \{g \in \GL(V)\ |\ g \cdot I_S = I_S\}.
$$
\end{lemma}

\begin{proof}
Let us first prove the second equality. If we denote by $(I_S)_{\leq d}$ the subspace of $I_S$ consisting of polynomials of degree $\leq d$. Indeed, applying the above lemma to $U = (I_S)_{\leq d}$ and $W = \C[V]_{\leq d}$, it follows that 
$$
\{g \in \GL(V)\ | g \cdot (I_S)_{\leq d} \subseteq (I_S)_{\leq d}\} = \{g \in \GL(V)\ |\ g \cdot(I_S)_{\leq d} = (I_S)_{\leq d}
$$

Since $I_S = \cup_{d} (I_S)_{\leq d}$, the second equality follows. Note that we did not directly apply to $I_S \subseteq \C[V]$ because they are not finite-dimensional.

Now, let us prove the first equality. For $g \in \mathcal{G}_S$ and $f \in I_S$, we observe that for $s \in S$, $(g \cdot f) (s) = f(g^{-1} \cdot s) = 0$.  Thus $g \cdot f \in I_S$. This shows that $g \cdot I_S \subseteq I_S$. This shows $\subseteq$. 

For the reverse inclusion. Suppose $g \in \GL(V)$ is such that $g \cdot I_S \subseteq I_S$. Then, by the second equality, we know that $g \cdot I_S = I_S$ and hence $g^{-1} \cdot I_S = I_S$. Now, suppose $s \in S$. We want to show that $g \cdot s \in S$. For any $f \in I_S$, we have $f(g \cdot s) = (g^{-1} \cdot f) (s) = 0$ since $g^{-1} \cdot f \in I_S$. This means that $gS \subseteq S$. Moreover, suppose $v \notin S$. Then for some $f \in I_S$, we have $f(v) \neq 0$. Thus $(g\cdot f) (gv) = f(v) \neq 0$. Since $g \cdot f \in I_S$, we get that $gv \notin S$. Thus $gS^c \subseteq S^c$, where $S^c$ denotes the complement of $S$ in $V$. 

Since $gS \subseteq S$ and $g S^c \subseteq S^c$, we have $gS = S$ (because $g$ is invertible). Thus $g \in \mathcal{G}_S$, and this concludes the proof.
\end{proof}

The same proof gives the following statement.

\begin{lemma} \label{asdfs}
Suppose the zero locus of $(I_S)_{\leq d}$ is precisely $S$. Then the group of symmetries
$$
\mathcal{G}_S = \{g \in \GL(V)\ | g \cdot (I_S)_{\leq d} \subseteq (I_S)_{\leq d}\}
$$
\end{lemma}

\begin{proof}
Run the same argument as above, but even easier because the infinite dimensional issue doesn't arise.
\end{proof}

\begin{proof} [Proof of Lemma~\ref{Lgos-graded}]
Again run the same argument as above. The hypothesis of $S$ being a cone can be ignored. However, unless $S$ is a cone, the zero locus of $(I_S)_a$ cannot possibly equal $S$.
\end{proof}

Let us note that to invoke Lemma~\ref{asdfs} (or Lemma~\ref{Lgos-graded}) in any explicit situation, one has to find an appropriate $d$ (or $a$), which is not always an easy task.

Next, we will carry over these results to the setting of Lie algebras, culminating in a proof of Proposition~\ref{P-Liealg-all}. First, a focused introduction to Lie theory.

\subsubsection{Lightning introduction to Lie theory}
Let us start with an example. 
\begin{example}
Let $V$ be a complex vector space with basis $e_1,\dots,e_n$, and let the corresponding coordinate functions be denoted $x_1,\dots,x_n$. The group $\GL(V)$ consists of all invertible linear transformations from $V$ to itself, and can be identified with invertible $n \times n$ matrices using the chosen basis. Its Lie algebra $\gl(V)$ consists of all linear transformations from $V$ to itself, and so can be identified with $\Mat_{n}$. There is an {\em exponential map} ${\rm exp}: \gl(V) = \Mat_{n} \rightarrow \GL(V) = \GL_n$ that sends $M \mapsto {\rm exp}(M) = I + M + \frac{M^2}{2!} + \dots + \frac{M^n}{n!} + \dots$. 
\end{example}

With this example in mind, let us give some definitions. A Lie group $G$ is a smooth manifold (over the real numbers $\R$) which is also a group such that the multiplication map and inverse map are smooth. To a Lie group $G$, one associates a Lie algebra denoted $\Lie(G)$ or $\g$ (in general, we may use the corresponding fraktur letter to make notation less cumbersome). The Lie algebra $\Lie(G)$ is the space of all {\em left-invariant vector fields}, equipped with a bilinear operation called the Lie bracket. 

A vector field on $G$ is the assignment of a tangent vector to each point of $G$. Left multiplication by group elements allows us to identify the tangent space at any point with the tangent space at the identity element $e \in G$. A vector field is called left-invariant if the assigned tangent vectors at all the points are the same (with the identification mentioned above). Thus one can identify the space of left-invariant vector fields with the tangent space at identity. A curve on $G$ is called an integral curve for a vector field if the tangent vector of the curve at every point agrees with the vector field. For a left-invariant vector field $M$, the curve ${\rm exp}(tM)$ is the unique integral curve for $M$ that passes through $e \in G$ at $t = 0$. In particular, $\frac{d}{dt} {\rm exp}(tM)|_{t=0} = M$ and this will be useful to us.

In the case of $\GL(V)$, let us reconcile the abstract definitions with the concrete ones in the above definition. Note that $\GL(V) = \GL_n$ is an open subset of $\Mat_{n}$. Hence, the tangent space at the identity can be identified with $\Mat_{n}$. Thus $\gl(V)$, the Lie algebra of $\GL(V)$ can be identified with $\Mat_{n}$. The abstract exponential map coincides with the concrete description given in the example above. 

For any Lie subgroup $H$ of any Lie group $G$, its Lie algebra is
$$
\Lie(H) = \{M \in \Lie(G)\ |\ {\rm exp}(tM) \in H\  \forall t\}.
$$

Many Lie groups occur as subgroups of $\GL(V)$, and these are often called matrix groups, and one can work extremely concretely in the setting of matrix groups. However, not every Lie group is a matrix group. For our purposes, the abstract point of view is elegant and helps us in the theoretical results, and the concrete description is more conducive for computations which is our main goal.

%For a Lie subgroup $G \subseteq \GL(V)$, its Lie algebra will be
%$$
%\Lie(G) = \{M \in \gl(V)\ |\ {\rm exp}(tM) \in G \ \forall t\}.
%$$

For any smooth action of an Lie group $G$ on a (finite-dimensional) vector space $W$ by linear transformations, we get a smooth morphism of Lie groups $\rho: G \rightarrow \GL(W)$. On differentiating, we get a morphism of Lie algebras $d\rho: \Lie(G) \rightarrow \gl(W)$. In other words, we get an induced action of $\Lie(G)$ on $W$. Note that algebraic groups are Lie groups and algebraic actions of algebraic groups are smooth.

The exponential map commutes with this, i.e., for $M \in \Lie(G) \subseteq \gl(V)$, we have ${\rm exp}(d\rho \cdot M) = \rho ({\rm exp} (M))$. We will simply write ${\rm exp}(M)$ for $\rho({\rm exp}(M))$ whenever there is no possibility of confusion. 

The action of $\GL(V)$ on $V$ (by left multiplication) gives an action of $\GL(V)$ on $\C[V]$, $\C[V]_{\leq d}$ (polynomials of degree $\leq d$) and $\C[V]_a$ (homogeneous polynomials of degree $a$) by the formula $(g \cdot f) (v) = f(g^{-1} v)$. When we take $W = \C[V]_{\leq d}$ or $\C[V]_a$, the above discussion gives an action of the Lie algebra $\gl(V)$. Thus, we have an action of $\gl(V)$ on $\C[V]_a$ and hence on $\C[V] = \oplus_{a \in \N} \C[V]_a$. The Lie algebra $\gl(V)$ acts on $\C[V]$ by derivations, and this we described explicitly in Section~\ref{sec:explicit}

\subsubsection{Computing the Lie algebra of symmetries}
First a lemma. Suppose we have a Lie group $G$ acting on a vector space $W$ by linear transformations. Let $U$ be a linear subspace of $W$. Then let $\mathcal{G}_U := \{g \in G\ |\ g U = U\} = \{g \in G\ |\ gU \subseteq U\}$. The latter equality follows from Lemma~\ref{repeat}. We have:

\begin{lemma}
The Lie algebra $\Lie(\mathcal{G}_U) = \{M \in \Lie(G)\ |\ M \cdot U \subseteq U\}$.
\end{lemma}

\begin{proof}
Suppose $M \in \Lie(\mathcal{G}_U)$. Let $u \in U$. Then ${\rm exp}(tM) \cdot u \in U$ for all $t$ because ${\rm exp}(tM) \in \mathcal{G}_U$. Thus ${\rm exp}(tM) \cdot u$ is a smooth curve in $U$. For the vector space $W$, the tangent space at any point is $W$. For any smooth curve completely contained in $U$, it is clear that the tangent vectors at any point of the curve is also in $U$. Thus, we have $M \cdot u = \frac{d}{dt} ({\rm exp}(tM) \cdot u)|_{t = 0} \in U$. So, we conclude that $M \cdot U \subseteq U$.

Conversely, suppose $M \cdot U \subseteq U$. Then for $u \in U$, we have ${\rm exp}(tM) \cdot u = \lim_{n \rightarrow \infty} (\sum_{i=0}^n \frac{M^i}{i!}) \cdot u$. Since each $(\sum_{i=0}^n \frac{M^i}{i!}) \cdot u \in U$, the limit is also in $U$. Thus ${\rm exp}(tM) \cdot U \subseteq U$, which means that ${\rm exp}(tM) \in \mathcal{G}_U$. Hence, $M \in \Lie(\mathcal{G}_U)$.
\end{proof}

Applying the above lemma, we can formulate the Lie algebra versions of the results at the beginning of this appendix (i.e., Lemma~\ref{L-ode}, Lemma~\ref{asdfs} and Lemma~\ref{Lgos-graded}).

\begin{proposition} \label{P-Liealg-all2}
Let $S  \subseteq V$ be an algebraic subset, and let $\mathcal{G}_S$ denote its group of symmetries, and $\g_S$ its Lie algebra of symmetries. Let $I_S$ denote the ideal of polynomial functions in $\C[V]$ that vanish on $S$. Then we have
$$
\g_S = \{M \in \gl(V)\ |\ M \cdot I_S \subseteq I_S\}.
$$
Further, if $I_S$ is generated in degree $\leq d$, then we have 
$$
\g_S = \{M \in \gl(V)\ |\ M \cdot (I_S)_{\leq d} \subseteq (I_S)_{\leq d}\}.
$$
Moreover, if $S$ is a cone, then $I_S$ is graded, and for any $a \in \N$, we have
$$
\g_S \subseteq \{M \in \gl(V)\ |\ M \cdot (I_S)_a \subseteq (I_S)_a\}.
$$
Finally, if the zero locus of $(I_S)_a$ is precisely the cone $S$, then we have equality.

\end{proposition}

The last part of the above proposition is precisely Proposition~\ref{P-Liealg-all}.

%The above proposition is extremely useful because it reduces the computation of the connected group of symmetries (or rather the Lie algebra of symmetries) to a linear algebraic computation. In any particular instance, one could even run a computer program to compute the Lie algebra of symmetries. In the later sections, we will apply this technique for one of our main results (Theorem~\ref{theo:cgos}), i.e., the computation of the connected group of symmetries for the case of $S = \SING_{n,m}$.

\section{Missing proofs for Section~\ref{sec:van}} \label{App.rep}
We will need the representation theory of $\GL_m \times \GL_n \times \GL_n$. In particular, an understanding of weights and highest weight vectors will be needed. We will recall the necessary background. 

For this section, let $I \subseteq \C[\Mat_{n}]$ denote the ideal of polynomial functions that vanish on $\SING_{n,m}$. The first observation is that since $\SING_{n,m}$ is stable under the action of $\GL_m \times \GL_n \times \GL_n$, so is $I$. Let us now make this more precise.

We have an action of $\GL_m \times \GL_n \times \GL_n$ on $\Mat_{n}^m$ given by 
\begin{equation}
(A,B,C) \cdot (X_1,\dots,X_m) = (\sum_{j=1}^m a_{1j}BX_j C^{t}, \sum_j a_{2j} BX_j C^{t}, \dots, \sum_j a_{mj} BX_jC^t),
\end{equation}
where $a_{ij}$ denotes the $(i,j)^{th}$ entry of $A$. While this is the most natural action, we will use a slight variant of this action which will make easier some later arguments. Consider the Cartan involution\footnote{Differentiating this gives the Cartan involution on Lie algebras described in Section~\ref{sec:explicit}.} $\theta: \GL_k \rightarrow \GL_k$ given by $\theta(A) = (A^{-1})^t$. We will twist the above action with the Cartan involution of each of the $\GL$'s. In the below formula, we will write $A' = \theta(A)$, $B' = \theta(B)$ and $C' = \theta(C)$. Moreover, we will write $a'_{ij}$ to denote the $(i,j)^{th}$ entry of $A'$. The action of $\GL_m \times \GL_n \times \GL_n$ on $\Mat_{n}^m$ we will use is given by
\begin{equation} \label{dual-action}
(A,B,C) \cdot (X_1,\dots,X_m) = (\sum_{j=1}^m a'_{1j}B'X_j (C^{t})', \sum_j a'_{2j} B'X_j (C^{t})', \dots, \sum_j a'_{mj} B'X_j (C^t)'),
\end{equation}

Let us now justify briefly why we use the second action instead of the first one. For $k \in \N$, there is a natural action of $\GL_k$ on $\C^k$ (viewed as column vectors) by left multiplication. This gives the contragredient action of $\GL_k$ on $(\C^k)^*$ (for $g \in \GL_k$ and $\zeta \in (\C^k)^*$, the element $g \cdot \zeta \in (\C^k)^*$ is defined by $g \cdot \zeta (v) = \zeta(g^{-1} \cdot v)$ for $v \in \C^k$). This is the canonical action of $\GL_k$ on $(\C^k)^*$. Thus, we have an action of $\GL_m \times \GL_n \times \GL_n$ on $(\C^m)^* \otimes (\C^n)^* \otimes (\C^n)^*$ where each $\GL$ acts on the corresponding tensor factor. 

Let $e_1,\dots,e_k$ denote the standard basis for $\C^k$ and let $e_1^*,\dots,e_k^*$ denote the corresponding dual basis of $(\C^k)^*$. We identify $\Mat_{n}^m$ with $(\C^m)^* \otimes (\C^n)^* \otimes (\C^n)^*$ as follows. $X = (X_1,\dots,X_m) \leftrightarrow \sum_{i,j,k} (X_i)_{j,k} e_i^* \otimes e_j^* \otimes e_k^*$. With this identification, the action of $\GL_m \times \GL_n \times \GL_n$ on $(\C^m)^* \otimes (\C^n)^* \otimes (\C^n)^*$ agrees with the one in Equation~\ref{dual-action} above. 

The advantage of this action is that the induced action on the coordinate ring $\C[\Mat_{n}^m]$ is a ``polynomial'' representation. This is notationally advantageous for many reasons -- in particular that polynomial representations of $\GL_m \times \GL_n \times \GL_n$ are indexed by triples of partitions (more details later). The action on polynomial functions is as follows. For $g \in \GL_m \times \GL_n \times \GL_n$, and $f \in \C[\Mat_{n}^m]$, we have $g \cdot f$ is the polynomial defined by the formula $g \cdot f (X) = f(g^{-1} \cdot X)$ for $X \in \Mat_{n}^m$.

%\begin{remark}
%If one identifies $\Mat_{n}^m$ with $\C^m \otimes \C^n \otimes \C^n$ in the standard way, then the action of $\GL_m \times \GL_n \times \GL_n$ is the natural one, i.e., for $(A,B,C) \in \GL_m \times \GL_n \times \GL_n$ and $T = \sum_{i} v_i \otimes w_i \otimes z_i$, we have
%$$
%(A,B,C) \cdot T = \sum_i Av_i \otimes Bw_i \otimes Cz_i.
%$$
%\end{remark}

\begin{lemma}
If $X = (X_1,\dots,X_m) \in \SING_{n,m}$, and $g \in \GL_m \times \GL_n \times \GL_n$, then $g \cdot X \in \SING_{n,m}$.
\end{lemma}

\begin{proof}
Let $g = (A,B,C)$ as above. We leave it to the reader to check that $\spa(g \cdot X) = B' (\spa(X)) (C^t)'$. Since $B$ and $C$ are invertible, $\spa(X)$ contains a non-singular matrix if and only if $\spa(g \cdot X)$ contains a non-singular matrix. 
\end{proof}

\begin{corollary}
The ideal $I \subseteq \C[\Mat_{n}^m]$ is $\GL_m \times \GL_n \times \GL_n$ stable. 
\end{corollary}

\begin{proof}
Suppose $f \in I$, and $g \in \GL_m \times \GL_n \times \GL_n$. Then for any $X \in \SING_{n,m}$, we have $(g \cdot f) (X) = f(g^{-1} \cdot X) = 0$ since $g^{-1} \cdot X \in \SING_{n,m}$ by the above lemma. Thus, $g \cdot f$ vanishes on $\SING_{n,m}$ and hence $g \cdot f \in I$ as required.
\end{proof}

We want to prove Proposition~\ref{P-ideal}. The first part of the proposition follows since $S = \SING_{n,m}$ is a cone. For, the second part, it is simple to see that $I_d = 0$. Let $x^{(k)}_{i,j}$ denote the coordinate function corresponding to the $(i,j)^{th}$ entry of the $k^{th}$ matrix, so $\C[\Mat_{n}^m] = \C[x^{(k)}_{i,j}\ |\ 1 \leq i,j \leq n, 1 \leq k \leq m]$.

\begin{proof} [Proof of Proposition~\ref{P-ideal}, part (2)]
Take $f \in \C[\Mat_{n}^m]_d$. Write $f$ as a sum of monomials. Suppose a monomial $m = \prod (x^{(k)}_{ij})^{a^{(k)}_{ij}}$ occurs in $f$ with non-zero coefficient (In particular, $\sum a^{(k)}_{ij} = d$). Consider the support of $m$, i.e., ${\rm Supp}(m) = \{(k,i,j)\ |\ a^{(k)}_{ij} >0\}$. The cardinality of ${\rm Supp}(m)$ is at most $d < n$. Let $X = (X_1,\dots,X_m)$ be such that $(X_k)_{ij} = 1$ if $(k,i,j) \in {\rm Supp}(m)$ and $0$ otherwise. Then the number of non-zero entries in any linear combination $\sum_i c_i X_i$ is at most $d$. Any matrix with at most $d$ non-zero entries is singular, so this means that $X \in \SING_{n,m}$\footnote{In fact $X \in \NSING_{n,m}$ -- this is not hard (for example it follows from the shrunk subspace criterion, see \cite{GGOW16,IQS}).}. Moreover observe that by construction $f(X) = $ coefficient of $m$ in $f$ which is non-zero, so $f \notin I_d$.  
\end{proof}

The action of $\GL_m \times \GL_n \times \GL_n$ on polynomials preserve degree, so the homogeneous polynomials of degree $d$, i.e., $\C[\Mat_{n}^m]_d$ is a subrepresentation. From the above corollary, we get that $I_n \subseteq \C[\Mat_{n}^m]_n$ is a subrepresentation. The group $\GL_m \times \GL_n \times \GL_n$ is reductive, so its representations can be decomposed as a direct sum of irreducible representations. Thus to understand $I_n$, we will have to understand the irreducibles that make up $I_n$ and their multiplicities.We will need some representation theoretic results, and we will be very brief, picking up only those results that are necessary. 

Irreducible polynomial representations of $\GL_m \times \GL_n \times \GL_n$ are indexed by triples of partitions $(\lambda,\mu,\nu)$ (where $\lambda$ has at most $m$ parts and $\mu,\nu$ have at most $n$ parts). A partition $\lambda$ is a (finite) decreasing sequence of positive numbers $(\lambda_1,\lambda_2,\dots,\lambda_k)$. We write $\lambda \vdash d$ if $\sum_i \lambda_i = d$. We will denote the irreducible representation corresponding to $(\lambda,\mu,\nu)$ by $S_{\lambda,\mu,\nu}$. An explicit description is given by Schur functors. For any partition $\pi$, denote by $S_{\pi}$ the Schur functor corresponding to $\pi$ as defined in \cite{Fulton}. $S_{\pi}$ is a functor from the category of vector spaces to itself. We refer to \cite{Fulton, Weyman} for an extensive introduction. For us, it suffices to remark that
$$
S_{\lambda,\mu,\nu} = S_{\lambda}(\C^m) \otimes S_{\mu}(\C^n) \otimes S_{\nu}(\C^n).
$$

%In this section, we will use a slightly different action of $\GL_m \times \GL_n \times \GL_n$ on $\Mat_{n}^m$, i.e., we will compose the action with a twist. This doesn't change

\begin{remark}
For a reductive group $G$, let $\{V_i:i \in I\}$ denote the irreducible representations. Then for any representation $V$, it can be decomposed as a direct sum of irreducibles. Such a decomposition is not always unique. Let $E_i$ denote the isotypic component w.r.t $V_i$, i.e., the sum of all subrepresentations of $V$ that are isomorphic to $V_i$. Then the isotypic decomposition $V = \oplus_{i \in I} E_i$ is unique. Further, each $E_i = V_i^{\oplus m_i}$, and $m_i$ is called the multiplicity of $V_i$ in $V$.
\end{remark}

Consider the decomposition of $\C[\Mat_{n}^m]_d$ into isotypic components 
$$
\C[\Mat_{n}^m]_d = \bigoplus_{\lambda,\mu,\nu \vdash d} E_{\lambda,\mu,\nu}
$$
where $E_{\lambda,\mu,\nu}$ is the isotypic component corresponding to $S_{\lambda,\mu,\nu}$. We have $E_{\lambda,\mu,\nu} = S_{\lambda,\mu,\nu}^{a_{\lambda,\mu,\nu}}$ where $a_{\lambda,\mu,\nu} \in \N$ are the celebrated Kronecker coefficients. This is actually one of many equivalent ways to define Kronecker coefficients. We now focus on degree $n$ polynomials.

\begin{lemma} \label{L-kron1}
We have
$$
a_{\lambda, 1^n,1^n} = \begin{cases} 1 & \text{ if $\lambda = (n)$}; \\
0 & \text{ otherwise.}
\end{cases}
$$
\end{lemma}

\begin{proof}
To see this, we have to understand Kronecker coefficients from the symmetric groups perspective. For a partition $\lambda \vdash n$, denote by $T_{\lambda}$ the corresponding representation of $S_n$. Then the decomposition of the tensor product $T_{\lambda} \otimes T_{\mu}$ into irreducibles is described by Kronecker coefficients, i.e.,

$$
T_{\lambda} \otimes T_{\mu} = \bigoplus_{\nu} T_{\nu}^{a_{\lambda,\mu,\nu}}.
$$

Note that for $\lambda = 1^n$, $T_{\lambda}$ corresponds to a $1$-dimensional representation which is called the sign representation. With this explicit description, one can deduce that $T_{\lambda} \otimes T_{1^n} = T_{\lambda^{\dag}}$ where $\lambda^{\dag}$ denotes the conjugate partition of $\lambda$. Thus, $a_{\lambda,1^n,1^n} = 0$ unless $\lambda = (1^n)^\dag = (n)$ and in the latter case, we have $a_{(n),1^n,1^n} = 1$. 
\end{proof}

\begin{lemma}
For partitions $\lambda,\mu,\nu \vdash n$, the isotypic component $E_{\lambda,\mu,\nu} \cap I_n = \phi$ if $\mu \neq (1,1,\dots,1)$.
\end{lemma}

\begin{proof}
First, let $T_k \subseteq \GL_k$ denote the standard torus, i.e., the all invertible diagonal matrices. Then $T_k$ is a maximal torus for $\GL_k$. For our purposes, $T = T_m \times T_n \times T_n$ is a maximal torus of $\GL_m \times \GL_n \times \GL_n$.

%A function $f \in \C[\Mat_{n}^m]$ such that $(t_1,t_2,t_3) \cdot v = \lambda(t_1) \mu(t_2) \nu(t_3) v$, where $(t_1,t_2,t_3) \in T _m \times T_n \times T_n = T$ are diagonal is called a weight vector of weight $(\lambda,\mu,\nu)$. In the above, $\lambda(t_1) = \sum_{i=1} \lambda_i (t_1)_{ii}$. Note that what we have described above is only a weight vector. A highest weight vector satisfies additional conditions, but we will not need them here.

A weight vector for the action of $\GL_m \times \GL_n \times \GL_n$ is simply a weight vector for $T$, which is a torus. We have already discussed weight vectors for actions of tori. Further, the characters of $T = T_m \times T_n \times T_n$ can be identified with $\Z^m \times \Z^n \times \Z^n$, and so a triple of partitions $(\lambda,\mu,\nu)$ (where $\lambda$ has at most $m$ parts and $\mu,\nu$ have at most $n$ parts) can be identified with a character. 

If $E_{\lambda,\mu,\nu} \cap I_n \neq \phi$, then $I_n$ has a subrepresentation isomorphic to $S_{\lambda,\mu,\nu}$, then it has a highest weight vector of weight $(\lambda,\mu,\nu)$ (since $S_{\lambda,\mu,\nu}$ is generated by such a highest weight vector -- a basic fact). Thus, to show that $E_{\lambda,\mu,\nu} \cap I_n = \phi$, it suffices to show that {\em all} weight vectors of weight $(\lambda,\mu,\nu)$ are not in $I_n$.

Let us consider all weight vectors of weight $(\lambda,\mu,\nu)$ in $\C[\Mat_{n}^m]_n$. Let $x^{(k)}_{ij}$ denote the coordinate function corresponding to the $(i,j)^{th}$ entry of the $k^{th}$ matrix $X_k$. Then its weight is $(\delta_k,\delta_i,\delta_j)$, where $\delta_a = (0,\dots,0,\underbrace{1}_a,0,\dots,0)$. Thus for a monomial $\prod (x^{(k)}_{ij})^{n^{(k)}_{ij}}$, its weight is $\sum n^{(k)}_{ij} (\delta_k,\delta_i,\delta_j)$.

The collection of all weight vectors of weight $(\lambda,\mu,\nu)$ in $\C[\Mat_{n}^m]$ is a linear subspace spanned by monomials of weight $(\lambda,\mu,\nu)$. Observe that since $\mu \vdash n$ and $\mu \neq (1,1,\dots,1)$, we have that $\mu_n = 0$. Thus, any monomial of degree $n$ whose weight is $(\lambda,\mu,\nu)$ does not depend on the last rows of the matrices. 

Thus, any weight vector $f$ of weight $(\lambda,\mu,\nu)$ is a linear combination of monomials all of which do not involve the last rows of the matrices. Let $U \subseteq \Mat_{n}^m$ be the subspace of tuples of matrices whose last row is zero. Then the weight vector $f$ is a nonzero polynomial on $U$, and $U \subseteq \SING_{n,m}$ (in fact $U \subseteq \NSING_{n,m}$). Thus, the weight vectors of weight $(\lambda,\mu,\nu)$ are not in $I_n$. 
\end{proof}

\begin{corollary}
The multiplicity of $S_{(n),1^n,1^n}$ in $\C[\Mat_{n}^m]_n$ is one, and this subrepresentation is equal to $I_n$.
\end{corollary}

\begin{proof}
First, we note that $I_n$ is a direct sum of irreducible subrepresentations of $\C[\Mat_{n}^m]$. Second, we note that all the isotypic components other than $E_{(n),1^n,1^n}$ do not intersect $I_n$. This is because, for any other choice of $(\lambda, \mu, \nu)$ for which $a_{\lambda,\mu,\nu} >0$, we have either $\mu \neq 1^n$ or $\nu \neq 1^n$ by Lemma~\ref{L-kron1}. If $\mu \neq 1^n$, then the above lemma tells us that the isotypic component does not intersect $I_n$. If $\nu \neq 1^n$, the argument is similar (In the proof of the above lemma, one would use the last column being zero rather than the last row). Thus $I_n \subseteq E_{(n),1^n,1^n}.$ Now, since $a_{(n),1^n,1^n} = 1$ by Lemma~\ref{L-kron1}, we know that $E_{(n),1^n,1^n}$ is irreducible and has no proper subrepresentations. Clearly $I_n \neq \{0\}$ since $\det(X_1) \in I_n$, so $I_n = E_{(n),1^n,1^n}$, which comprises of the unique copy of $S_{(n),1^n,1^n}$ in $\C[\Mat_{n}^m]_n$.
\end{proof}

\begin{proof} [Proof of Proposition~\ref{P-ideal}, part (3)]
This can be seen in many ways, some very explicit. However, we will take a short route out by making the following observation. Consider the action of $\SL_n \times \SL_n \subseteq \GL_n \times \GL_n \subseteq \GL_m \times \GL_n \times \GL_n$ by left-right multiplication on $\Mat_{n}^m$. Then, the invariant polynomials of degree $n$ are precisely the isotypic component corresponding to $S_{(n),1^n,1^n}$ (see for example \cite[Proposition~4.1]{Visu}), which by the above corollary is precisely $I_n$.

There has been much work on the ring of invariants for the $\SL_n \times \SL_n$ action. It is a special case of a semi-invariant ring of quivers, in particular for the generalized Kronecker quiver. Such semi-invariants have explicit determinantal descriptions, an important and non-trivial result shown simultaneously and independently by three groups of researchers (see \cite{DW,SVd,DZ}). From this description, we get that the invariants of degree $n$ for the action of $\SL_n \times \SL_n$ are spanned by polynomials of the form $\det(\sum_i c_i X_i)$ (see \cite{DM,IQS}). This completes the proof.
\end{proof}

\subsection{Ideal of polynomials vanishing on $\NSING_{n,m}$}
We note that all the arguments above for understanding the ideal of polynomials vanishing on $\SING_{n,m}$ work equally well for $\NSING_{n,m}$, and one obtains:

\begin{lemma} \label{L-N-ideal}
Consider $\NSING_{n,m} \subseteq V = \Mat_{n}^m$, and let $J \subseteq \C[V]$ be the ideal of polynomial functions that vanish on $\NSING_{n,m}$. Then,
\begin{enumerate}
\item $J$ is graded;
\item $J_a$ is empty if $a < n$;
\item $J_n = \spa (\det(\sum_i c_iX_i): c_i \in \C).$
\end{enumerate}
\end{lemma}

\section{Positive characteristic} \label{app.pos.char}
In this section, we will point out the parts of the paper that require characteristic zero, and how to make the requisite modifications for the statements to hold in. Let $K$ be an algebraically closed field of arbitrary characteristic. The first issue comes with the use of Lie algebras. Lie algebras are a little trickier to define because of the lack of derivatives when working over $K$. Nevertheless, one can define the Lie algebra $\Lie(G) = \g$ of an algebraic group $G$ as the space of all derivations of $K[G]$. As a vector space one can identify this with the tangent space at the identity element. Next, one does not have an exponential map that allows us to pass back from the Lie algebra to the group. 
%However, it is true that $\dim \g \geq \dim G$, and a strict inequality is very well possible. If the group $G$ is smooth, then we do have equality.

Now, let $V$ be a vector space over $K$ and let $S \subseteq V$ be a subvariety. Let $\mathcal{G}_S$ denote its group of symmetries, $\mathcal{G}_S^\circ$ its identity component and $\g_S$ its Lie algebra. Let $I_S$ denote the ideal of polynomials in $K[V]$ vanishing on $S$. When we work over $K$ instead of $\C$, the proof of Lemma~\ref{Lgos-graded} clearly goes through and Proposition~\ref{P-Liealg-all} still remains true, for example by \cite[Proposition~10.31]{Milne}.

The other main issue in implementing this strategy is to understand the degree $n$ component of the ideal of polynomials vanishing on $\SING_{n,m}$, as done in the previous appendix.  Let $I = I_{\SING_{n,m}} \subseteq K[\Mat_{n}^m]$. On first glance, it looks like we used quite heavily the notion of complete reducibility for $\GL_m \times \GL_n \times \GL_n$ actions. But in fact, we can get away with far less. The first idea is to restrict our attention to $\SL_n \times \SL_n$. Clearly $I_n$ is an $\SL_n \times \SL_n$ subrepresentation. It need not break up as a direct sum of irreducibles, but will definitely have a composition series. Nevertheless, $I_n$ is a direct sum of weight spaces. We claim that the only highest weight vectors that can be in $I_n$ must have weight zero. Basically the argument we used in the previous section shows that any highest weight vector (for $\GL_n \times \GL_n$) in $I_n$ must have weight $((1,1,\dots,1), (1,1,\dots,1))$. Highest weight vectors for $\GL_n \times \GL_n$ are precisely the same as the highest weight vectors for $\SL_n \times \SL_n$, and the weight $((1,1,\dots,1), (1,1,\dots,1))$ for $\GL_n \times \GL_n$ corresponds to the zero weight for $\SL_n \times \SL_n$.

Now, let $\mathcal{X}$ denote the set of all weights for $I_n$ (w.r.t $\SL_n \times \SL_n)$ whose multiplicity is nonzero. It is well known (and easy to see) that the set of weights is stable under the action of the symmetric group $S_n$ (also known as the Weyl group). If $\mathcal{X}$ contains a non-zero weight, then it contains a non-zero dominant weight (using the action of $S_n$). Consider the collection of all non-zero dominant weights. Since this is a finite set, it has a maximal element w.r.t to the usual partial order ($\lambda \prec \mu$ if $\mu-\lambda$ is a sum of positive roots). Any weight vector for this maximal weight must be a highest weight vector! But this contradicts the discussion above, so $\mathcal{X}$ must be the singleton set $\{0\}$. 

In other words, $I_n$ is an $\SL_n \times \SL_n$ stable subspace of the zero weight space (in the space of degree $n$ polynomials on $\Mat_n^m$). Irreducible representations of $\SL_n \times \SL_n$ are indexed by their highest weights (holds true in all characteristics), and so any composition series for $I_n$ must only contain trivial representations. However, trivial representations do not have any self-extensions, so $I_n$ must be a direct sum of trivial representations. In other words, $I_n$ must be a subspace of the $\SL_n \times \SL_n$ invariants. On the other hand, we know (by \cite{DW}) that $\SL_n \times \SL_n$ invariants are spanned by polynomials of the form $\det(\sum_i c_i X_i)$, which are all clearly in $I_n$. This shows that Proposition~\ref{P-ideal} continues to hold.

Now, armed with the above results, one sees that the computation of $\g_S$ for $S = \SING_{n,m}$ is exactly the same, and we get $\g_S = \Lie(G_{n,m})$. However, the subgroups-subalgebras correspondence is not necessarily true over fields of positive characteristic, so we cannot immediately conclude that $G_S^\circ = G_{n,m}$. However, it is definitely clear that $G_{n,m} \subseteq G_S^\circ$. Suppose $G_S^\circ$ were a strictly larger algebraic group, then $\dim G_S^\circ > \dim G_{n,m} = \dim (\Lie(G_{n,m})) = \g_S$. Note that $\dim(G_{n,m}) = \dim(\Lie(G_{n,m}))$ follows from the fact that $G_{n,m}$ is smooth (or one can simply compare the dimensions to those in characteristic zero). Thus, we have $\dim G_S^\circ > \dim \g_S$, but this is a contradiction because $\dim \g_S$ is the dimension of the tangent space at identity for $G_S^\circ$, which is always at least $\dim G_S^\circ$ (see \cite[Proposition~1.37]{Milne}). Hence, we have $G_S^\circ = G_{n,m}$. The computation of the entire group of symmetries follows verbatim. The same arguments also compute the group of symmetries for $\NSING_{n,m}$ just as in the characteristic zero case.

The rest of the arguments are effectively the same. The coordinate subspaces in $\NSING_{n,m}$ and $\SING_{n,m}$ have exactly the same descriptions in terms of permutation free supports, and so as long as $\NSING_{n,m} \subsetneq \SING_{n,m}$, the latter cannot be a null cone. Finally, to show that $\NSING_{n,m} \subsetneq \SING_{n,m}$ for $n,m \geq 3$, we relied on an explicit example of $3$-tuple of $3 \times 3$ matrices which was in $\SING_{n,m}$, but not in $\NSING_{n,m}$. This example continues to hold in positive characteristic as well, which can be explicitly checked (it can also be derived as a special case of \cite[Proposition~1.8]{DM-explicit} for $p = 1$).

\end{document}